\documentclass{svmult}
  
\usepackage{mathptmx,hyperref}   
\usepackage{helvet} 
\usepackage{courier}
\usepackage{amsmath,amsfonts}
\usepackage{graphicx}
\usepackage[bottom]{footmisc}          
    
\usepackage{mathptmx}         
\usepackage{helvet}            
\usepackage{courier}            
\usepackage{type1cm}             
\usepackage{makeidx}             
\usepackage{graphicx}          
\usepackage{multicol}          
\usepackage[bottom]{footmisc}  
     
\usepackage{epsfig}     
\usepackage{psfrag,rotating}     
\usepackage{amssymb,floatflt,enumerate} 
\usepackage{amsmath,amscd,psfrag,leqno} 
\usepackage[mathscr]{eucal}  
\usepackage{shadethm}           
\usepackage{graphicx,subfigure}   
\usepackage{overpic,contour}       
\contourlength{0.3mm}  
 
\def\Beweisende{\square}            
\def\BewEnde{\hfill{\Beweisende}}

\def\phm{{\hphantom{-}}}

\newcommand{\mathtext}[1]{\mbox{\ \,#1\, \ }}

\def\RR{{\mathbb R}}

\def\NN{{\mathbb N}}
\def\CC{{\mathbb C}}

\def\Vkt#1{{\mathbf #1}}

\makeindex
\begin{document}

\title*{ 
A global approach for the redefinition \\of higher-order flexibility and rigidity
}

% Use \titlerunning{Short Title} for an abbreviated version of
% your contribution title if the original one is too long
\author{G. Nawratil}
\authorrunning{G. Nawratil}
% Use \authorrunning{Short Title} for an abbreviated version of
% your contribution title if the original one is too long
%\institute{$^1$IRCCyN, CNRS, France, \email{Philippe.Wenger@irccyn.ec-nantes.fr} \\
%$^1$University of Minho, Portugal, \email{pflores@dem.uminho.pt}}
\institute{
  Institute of Discrete Mathematics and Geometry \&  
	Center for Geometry and Computational Design, TU Wien, Austria \newline
  \email{nawratil@geometrie.tuwien.ac.at}}

%
% Use the package "url.sty" to avoid
% problems with special characters
% used in your e-mail or web address
%
\maketitle

\abstract{
The famous example of the double-Watt mechanism given by Connelly and Servatius 
raises some problems concerning the classical definitions of higher-order flexibility and rigidity, respectively, as they attest the cusp configuration of the mechanism a 
third-order rigidity, which conflicts with its continuous flexion. 
Some attempts were done to resolve the dilemma but they could not settle the problem. 
As cusp mechanisms demonstrate the basic shortcoming of any local mobility analysis using higher-order constraints, we present a global approach inspired by Sabitov's finite algorithm for testing the bendability of a polyhedron, 
which allows us (a) to compute iteratively configurations with a higher-order flexion and (b) to come up with a proper redefinition of higher-order flexibility and rigidity. 
The presented approach is demonstrated on several examples (double-Watt mechanisms and Tarnai's Leonardo structure). Moreover, we determine  all configurations of a given 3-RPR manipulator with a third-order flexion and present a corresponding joint-bar framework of flexion order 23.
}


\section{Introduction}\label{sec:intro}

In this paper we give a redefinition of higher-order flexibility and rigidity of bar-joint
frameworks.  
Such a framework $G(\mathcal{K})$ consists of a knot set 
\begin{equation}
\mathcal{K}=\left\{X_{1},\ldots, X_w\right\}    
\end{equation}
and a graph $G$  on $\mathcal{K}$. 
A knot $X_i$ corresponds 
a rotational/spherical joint (without clearance) in the case of a planar/spatial framework. 
An edge connecting two knots corresponds to a bar. We denote the number of edges by $e$.

By defining the combinatorial structure of the framework as well as the lengths of the bars, which are assumed to be non-zero,    
the intrinsic geometry of the framework is fixed. 
In general the assignment of the intrinsic metric does not uniquely determine the embedding of the framework into the Euclidean space $\RR^d$, thus such a framework can have 
different incongruent realizations.

\subsection{Algebraic approach to rigidity theory}
\label{sec:algapproach}
The relation that two elements of the knot set are edge-connected can also be expressed algebraically. 
They are either quadratic constraints resulting from 
a squared distance of vertices (implied by an edge) or
linear condition, in the case that one of the pin-joints gets an ideal-point. There are further linear conditions steaming from the elimination of 
isometries\footnote{This are 6 linear constraints for $d=3$ and 3 linear constraints for $d=2$.}.  
In total this results in a system of $l$ algebraic equations $c_1=0,\ldots ,c_l=0$  in $m$ unknowns $z_1,\ldots ,z_m$,
which constitute an algebraic variety $V(c_1,\ldots,c_l)$. 
Note that $l$ equals $e+6$ in the spatial case and 
 $e+3$ in the planar one.
Moreover, $m$
equals for the planar case $2w$ and for the spatial one $3w$.

If $V(c_1,\ldots ,c_l)$ is positive-dimensional then the framework is flexible; otherwise rigid. 
The framework is called minimally rigid (isostatic) if the removal of any algebraic constraint (resulting from an edge) will make the framework flexible. 
In this case $m=l$ has to hold. Rigid frameworks, which are not isostatic, are called {\it overbraced} or {\it overconstrained} ($l>m$). 

If $V(c_1,\ldots ,c_l)$ is zero-dimensional, then each real solution corresponds to a realization $G(\Vkt X)$ of the framework for $\Vkt X=(\Vkt x_1,\ldots , \Vkt x_w)$. 
If there is exactly one real solution, then the framework is called globally rigid. 

We can compute in a realization the tangent-hyperplane to each of the hypersurfaces $c_i=0$ in $\RR^m$ for $i=1,\ldots ,l$. Note that this is always possible as 
all hypersurfaces are either hyperplanes or regular hyperquadrics. 
The normal vectors of these tangent-hyperplanes constitute the columns of a  $m\times l$ matrix  $\Vkt R_{G(\Vkt X)}$, which is also known as {\it rigidity matrix} 
of the realization $G(\Vkt X)$; i.e.\
\begin{equation}\label{eq:rigidityM}
\Vkt R_{G(\Vkt X)}=
\begin{pmatrix}
    \frac{\partial c_1}{\partial z_1} &   \frac{\partial c_2}{\partial z_1} & \ldots &  \frac{\partial c_l}{\partial z_1} \\
    \frac{\partial c_1}{\partial z_2} &   \frac{\partial c_2}{\partial z_2} & \ldots &  \frac{\partial c_l}{\partial z_2} \\
    \vdots & \vdots  & \ddots  & \vdots \\
    \frac{\partial c_1}{\partial z_m} &   \frac{\partial c_2}{\partial z_m} & \ldots &  \frac{\partial c_l}{\partial z_m}
\end{pmatrix}.
\end{equation}
If its rank is $m$ then the realization is infinitesimal rigid otherwise it is infinitesimal flexible; i.e. the hyperplanes 
have a positive-dimensional affine subspace in common. Therefore the intersection multiplicity of the $l$ hypersurfaces is at least two in a 
shaky realization. Therefore a shaky configuration can also be seen as the limiting case where at least two realizations of a framework coincide \cite{stachel_wunderlich,stachel_between}.

Clearly, by using the rank condition $rk(\Vkt R_{G(\Vkt X)})<m$ one can also characterize all shaky realizations $G(\Vkt X)$ algebraically by  
the affine variety $V_1$. This so-called shakiness variety is the zero set of the ideal generated by the polynomials $p_{1},\ldots, p_{\mu}$ which correspond to 
all  $\mu:= {{l}\choose{l-m}}$ minors of $\Vkt R_{G(\Vkt X)}$ of order $m\times m$.
Note that for minimally rigid framework $\mu=1$ holds, where the infinitesimal flexibility is given by $p_{1}:\,\, \det(\Vkt R_{G(\Vkt X)})=0$.

\section{Review on higher-order flexibility and rigidity}\label{rev:higher}

A first paper on the higher-order flexion of surfaces was written by Rembs \cite{rembs}.
In contrast first results on higher-order rigidity of surfaces date back to Efimov \cite{efimov}. 
An exhaustive treatment of higher-order flexion and rigidity of surfaces was done by Sabitov in \cite{sabitov}, 
in which also a section is devoted to discrete structures. 
Connelly gave a definition of $2^{nd}$-order flexibility and rigidity of frameworks in \cite{connelly80}.
Tarnai wrote a paper \cite{tarnai} on the definition of higher-order infinitesimal mechanisms, 
which seems to be more problematic than that of a framework due to the existence of non-analytic kinematic pairs. 
According to Stachel \cite{stachel_proposal} all these approaches to higher-order flexible frameworks  
can be unified to the so-called {\it classical} definition, which reads as follows:

\begin{definition} \label{def1}
A framework has a $n^{th}$-order flex if for each vertex $\Vkt x_i$ ($i=1,\ldots,w$) 
there is a polynomial function 
\begin{equation}\label{eq:flex}
\Vkt x_i':=\Vkt x_i+ \Vkt x_{i,1}t+\ldots + \Vkt x_{i,n}t^n \quad \text{with} \quad n>0
\end{equation}
such that
\begin{enumerate}
\item
the replacement of $\Vkt x_i$ by $\Vkt x_i'$ in the equation of the edge lengths gives stationary values 
of multiplicity $\geq n+1$ at $t=0$;
\item
the velocity vectors $\Vkt x_{1,1},\ldots ,\Vkt x_{w,1}$ do not originate from a rigid body motion (incl.\ standstill) 
of the complete framework; i.e.\ they are said to be non-trivial. 
\end{enumerate}
\end{definition}

\begin{remark}
Tarnai's definition relies on the power-series expansion of the elongation of the bar in terms of the displacement, but his definition is equivalent 
to Definition \ref{def1} (cf.\ \cite{tarnai}). Following an idea of Koiter, one can replace the bar elongation by the strain energy of the bars,  
which also results in an equivalent definition (cf.\ \cite{salerno}). 
Moreover, Kuznetsov \cite{kuznetsov} gave another definition of higher-flexibility, which relies on the Taylor expansion 
of the constrained equations of the framework. Without noticing it, exactly the same approach was used by Chen \cite{chen} 
to define the local mobility of a mechanism. It can be seen from \cite{rameau2015}, that the (identical) 
definitions of Kuznetsov and Chen are again equivalent with Definition \ref{def1}. 
\hfill $\diamond$
\end{remark}

Based on the notion of $n^{th}$-order flex given in Definition \ref{def1} one can define $n^{th}$-order rigidity as follows \cite{connelly80,servatius}:
\begin{definition} \label{def2}
A framework is $n^{th}$-order rigid if every $n^{th}$-order flex has $\Vkt x_{1,1},\ldots ,\Vkt x_{w,1}$ trivial as a $1^{st}$-order flex. 
\end{definition}
\begin{remark}
Clearly, in the context of Definition \ref{def1} one is only interested in the flex with maximal $n$; i.e.\ the framework has to be rigid of order $(n+1)$ according to Definition \ref{def2}.
\hfill $\diamond$
\end{remark}
\begin{figure}[t]
\begin{center}
\begin{overpic}
    [height=20mm]{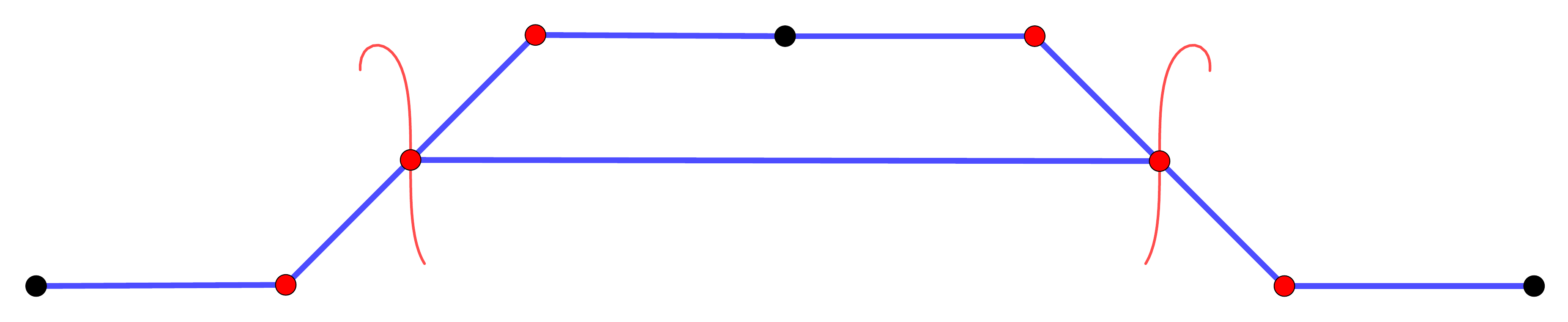}
\begin{scriptsize}
\put(22.5,10){$\Vkt x_1$}
\put(75.5,10){$\Vkt x_2$}
\end{scriptsize}     
  \end{overpic} 
\end{center}	
\caption{Double-Watt mechanism of Connelly and Servatius in its cusp configuration; i.e.\ the mechanism has an instantaneous standstill. The dimensions of each Watt mechanism are as follows: the arms have length $1$ and the coupler is of length $\sqrt{2}$.  The midpoints $\Vkt x_1$ and $\Vkt x_2$ of both couplers are connected by a bar of length $3$. 
}
  \label{fig1}
\end{figure}    
But the famous example of the double-Watt mechanism (cf.\ Fig.\ \ref{fig1}) given by Connelly and Servatius \cite{servatius} raises some problems concerning these Definitions \ref{def1} and \ref{def2}, 
as they attest this mechanism a $3^{rd}$-order rigidity in a certain configuration, which conflicts\footnote{\label{test} One expects from a proper definition that an $n^{th}$-order rigidity implies rigidity (cf.\ \cite{servatius}).} 
with its continuous flexibility. This configuration corresponds to a cusp in the configuration space \cite{servatius}, which was also pointed out by 
M\"uller's  study \cite{muller} of the mechanism from the perspective of kinematic singularities. 
Based on the latter work further examples of cusp mechanisms (even spatial ones) were given in \cite{lopez}.

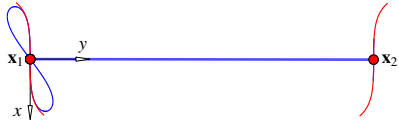
\begin{figure}[t]
\begin{center}
\begin{overpic}
    [height=16mm]{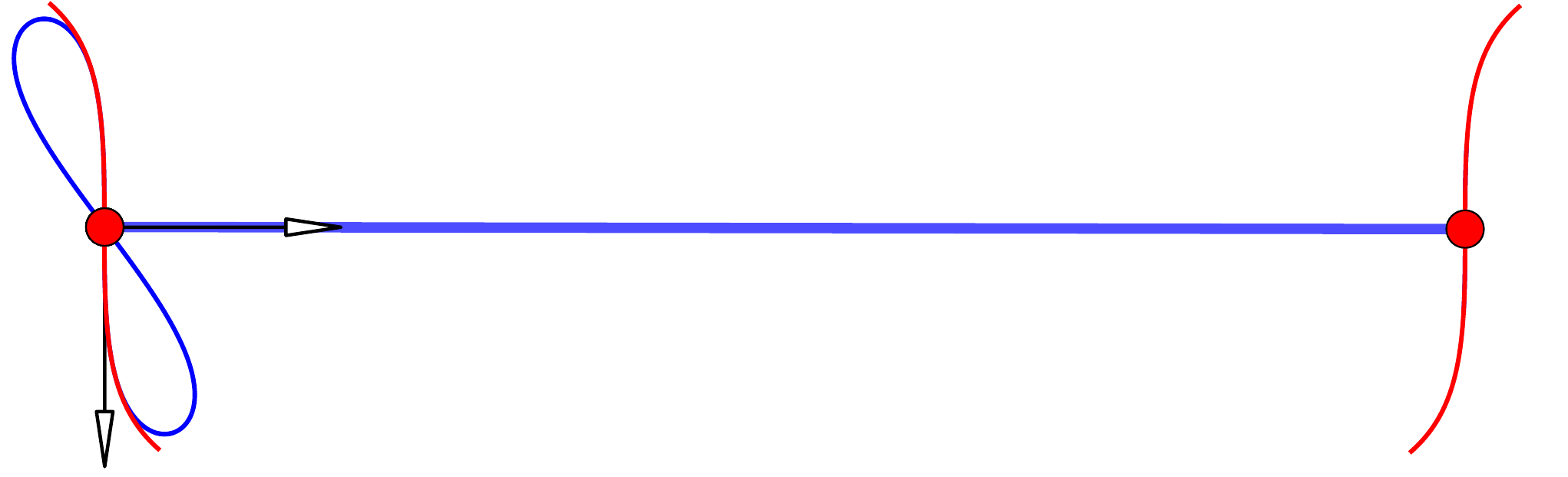}
\begin{scriptsize}
\put(2.5,1.5){$x$}
\put(19,18.5){$y$}
\put(1,15){$\Vkt x_1$}
\put(95.5,15){$\Vkt x_2$}
\end{scriptsize}     
  \end{overpic} 
\end{center}	
\caption{Reduction of the  double-Watt mechanism of Connelly and Servatius to a two-point guidance problem. 
}
  \label{fig1a}
\end{figure}    

\begin{example}\label{ex:watt1}
    In the following we present the analysis of the  double-Watt mechanism according to the method presented by Stachel in \cite{stachel_aim}.
    With respect to the coordinate system displayed in Fig.\ \ref{fig1a} the coupler-curve of the point $\Vkt x_1$ is given by the algebraic equation
    \begin{equation}\label{eq:watt_curve}
        x^6 + 3x^4y^2 + 3x^2y^4 + y^6 + 3x^4 + 6x^3y - 2x^2y^2 + 6xy^3 - 5y^4 - 6xy + 8y^2=0.
    \end{equation}
    We are interested in the branch where the $x$-axis is the tangent to the inflection point. It can be parametrized locally by means of  Puiseux series as:
    \begin{equation}\label{eq:x1watt}
        \Vkt x_1=\begin{pmatrix}
            \tau_1 \\
            \tfrac{1}{2}\tau_1^3+\tau_1^5+\tfrac{9}{4}\tau_1^7+\tfrac{13}{2}\tau_1^9+\ldots
        \end{pmatrix}.  
    \end{equation}
    Clearly, the path of $\Vkt x_2$ is obtained by reflection on the $x$-axis and by translation along the vector $(0,3)^T$ yielding:
    \begin{equation}\label{eq:x2watt}
        \Vkt x_2=\begin{pmatrix}
            \tau_2 \\
            3-\tfrac{1}{2}\tau_2^3-\tau_2^5-\tfrac{9}{4}\tau_2^7-\tfrac{13}{2}\tau_2^9-\ldots
        \end{pmatrix}.  
    \end{equation}
    Thus we end up with a two-point guidance problem, where the time dependence of $\tau_i$ is set up by 
    \begin{equation}\label{eq:taus}
        \tau_i=v_{i,1}t+v_{i,2}t^2+v_{i,3}t^3+\ldots.
    \end{equation}
    Now the $v_{i,j}$ have to be adjusted in order to fulfill
\begin{equation}\label{eq:F}
    F:=\|\Vkt x_2(\tau_2)-\Vkt x_1(\tau_1)\|^2-3^2=o(t^{n})
\end{equation}
for a $n^{th}$-order flexibility at $t=0$. 
We substitute Eq.\ (\ref{eq:taus}) into Eq.\ (\ref{eq:F}) and consider the coefficients $f_i$ of $t^i$ in the resulting expression. We get $f_1=0$ 
and $f_2=(v_{1,1}-v_{2,1})^2$. Setting  $v_{2,1}=v_{1,1}$ we get $f_3=-6v_{1,1}^3$. 
This means with $v_{1,1}\neq 0$ it is only flexible of $2^{nd}$-order implying $3^{rd}$-order rigidity. \hfill $\diamond$
\end{example}

Two attempts are known to the author to resolve the dilemma (cf.\ Footnote \ref{test}):  
Gaspar and Tarnai \cite{gaspar} suggested to use fractional exponents which corresponds to the replacement of Eq.\ (\ref{eq:flex}) by
\begin{equation}\label{eq:fractional}
\Vkt x_i':=\Vkt x_i+ \Vkt x_{i,1}t+\Vkt x_{i,\tfrac{3}{2}}t^{\tfrac{3}{2}}+ \Vkt x_{i,2}t^2 +\Vkt x_{i,\tfrac{5}{2}}t^{\tfrac{5}{2}}\ldots + \Vkt x_{i,n}t^n,
\end{equation}
where $\Vkt x_{1,1},\ldots ,\Vkt x_{w,1}$ is non-trivial. 
This solved the particular problem for the cusp configuration of the double-Watt mechanism but not the parametrization problem according to \cite{leonardo}, where it is also written that ``{\it a very promising approach was presented recently by \cite{stachel_proposal}}''.

This approach of Stachel follows the more general notation of $(k,n)$-flexibility suggested by Sabitov \cite{sabitov} which replaces Eq.\ (\ref{eq:flex}) 
by
\begin{equation}\label{eq:flexk}
\Vkt x_i':=\Vkt x_i+ \Vkt x_{i,k}t^k+\ldots + \Vkt x_{i,n}t^n \quad \text{with} \quad n\geq k>0
\end{equation}
where $\Vkt x_{1,k},\ldots ,\Vkt x_{w,k}$ is non-trivial. In addition Eq.\ (\ref{eq:flexk}) has to represent 
an irreducible flex; this means that  Eq.\ (\ref{eq:flexk}) does not result from 
a polynomial parameter substitution of a lower-order flex.

\begin{example}\label{ex:watt2}
Continuation of the double-Watt mechanism:
According to the notation of Eq.\ (\ref{eq:flexk}) the 
double-Watt mechanism in the cusp configuration is $(1,2)$-flexible but not $(1,3)$-flexible (cf.\ Example \ref{ex:watt1}). 
Therefore we set $v_{1,1}$ and continue Example \ref{ex:watt1} by considering $f_4=(v_{1,2}-v_{2,2})^2$. We set $v_{2,2}=v_{1,2}$ and get $f_5=0$. Moreover, for $f_6$ we obtain the expression $-6v_{1,2}^3+v_{1,3}^2-2v_{1,3}v_{2,3}+v_{2,3}^2$, which can be solved for\footnote{Note that the $\pm$ sign corresponds to the two ways out of the cusp configurations. \label{fn:pm}} $v_{2,3}=v_{1,3}\pm \sqrt{6v_{1,2}^3}$ showing $(2,6)$-flexibility. Moreover, we can proceed in this way (i.e.\ solving $f_i=0$ for $v_{2,i-3}$ for $i>6$) implying  $(2,\infty)$-flexibility. 

We only have to check that the $(2,\infty)$-flexibility was not obtained by the $(1,2)$-flexibility by a polynomial parameter substitution of the form
\begin{equation}\label{eq:polysubs}
    t=\overline{t}^p(a_0+a_1\overline{t}+a_2\overline{t}^2+\ldots) 
\end{equation}
with $a_0\neq 0$ and $p>1$. 
For $p=2$ we get $\overline{f}_1=\overline{f}_2=\overline{f}_3=0$. $\overline{f}_4=a_0^2(v_{1,1}-v_{2,1})^2$ implies $v_{2,1}=v_{1,1}$. Then $\overline{f}_5=0$ and $\overline{f}_6=-6a_0^3v_{1,1}^3$. Therefore the substitution turns the $(1,2)$-flexibility into a reducible $(2,5)$-flexibility. As a consequence the $(2,\infty)$-flexibility has to be an irreducible flex. \hfill 
$\diamond$ 
\end{example}

\begin{remark}
    Note that the substitution of Eq.\ (\ref{eq:polysubs}) into Eq.\ (\ref{eq:fractional}) for $p=2$ yields the  $(2,\infty)$-flexibility of Stachel discussed in the last example. Therefore, Stachel's approach also includes the one of Gaspar and Tarnai \cite{gaspar}. \hfill $\diamond$
\end{remark}
Stachel's proposal was only presented within the Tensegrity Workshop in 2007 \cite{stachel_proposal}, 
but remained unpublished so far. According to Stachel \cite{stachel_private}, the reason for this is the example of another double-Watt mechanism, which is extended by a Kempe-mechanism (cf.\ Fig.\ \ref{fig2}), presented in \cite{stachel_aim}, as no unique flexion order can be identified with his proposed definition. 
Therefore the problem is not yet settled.

\begin{example}\label{ex:watt3}
Stachel's double-Watt framework: In the following we also give this example of Stachel 
where the second Watt-mechanism is just a translation of the first one (see Fig.\ \ref{fig:doubleStachel})  by the vector $(0,3)^T$. Thus we get for the path of $\Vkt x_2$ the following parametrization
\begin{equation}\label{eq:x2stachel}
        \Vkt x_2=\begin{pmatrix}
            \tau_2 \\
            3+\tfrac{1}{2}\tau_2^3+\tau_2^5+\tfrac{9}{4}\tau_2^7+\tfrac{13}{2}\tau_2^9+\ldots
        \end{pmatrix}
    \end{equation}
for the interpretation as a two-point guidance problem, which is illustrated in Fig.\ \ref{fig:doubleStachel_reduced}. 
In this case the two-point guidance is in a branching configuration; i.e.\ it corresponds to a double point in the configuration space.

\begin{figure}[t]
\begin{center}
\begin{overpic}
    [height=20mm]{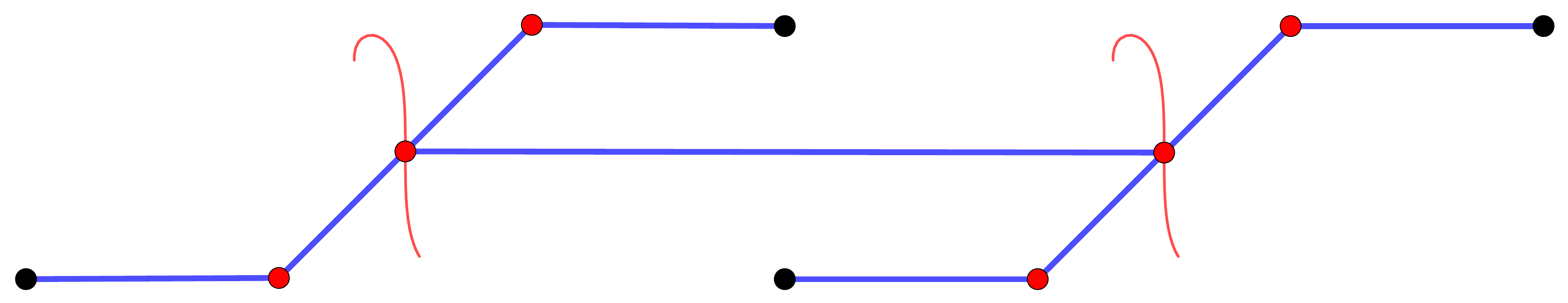}
  \end{overpic} 
\end{center}	
\caption{Double-Watt mechanism of Stachel in a branching configuration; i.e.\ it corresponds to a double point in the configuration space. 
}
  \label{fig:doubleStachel}
\end{figure}

\begin{figure}[t]
\begin{center}
\begin{overpic}
    [height=17.8mm]{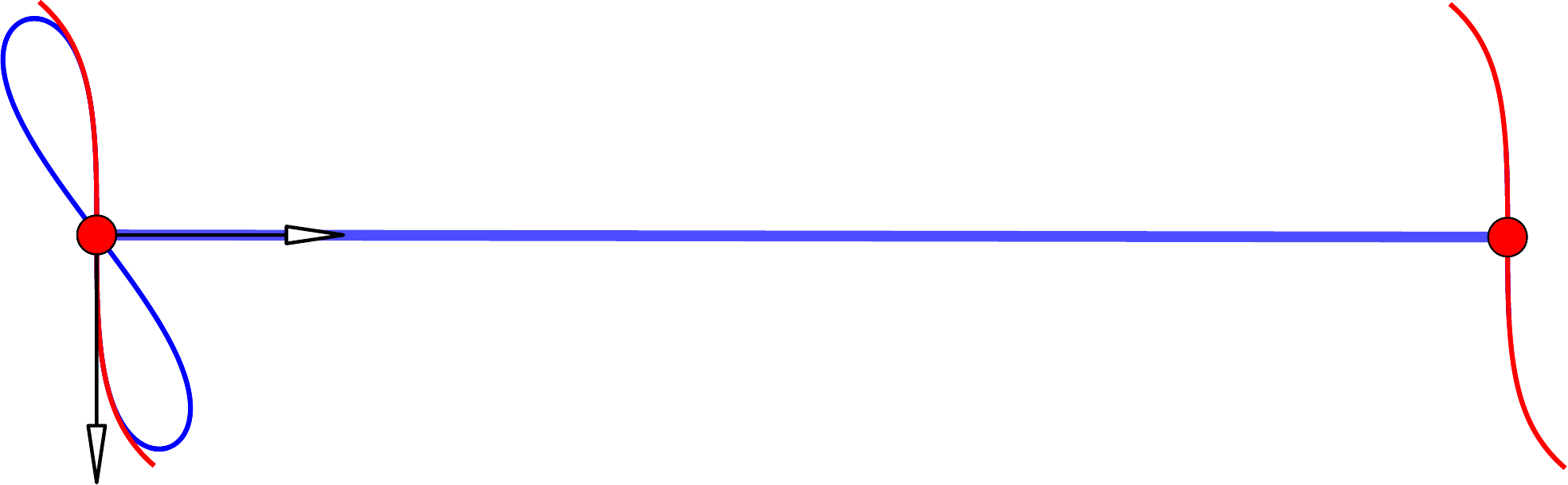}
\begin{scriptsize}
\put(3.5,1.){$x$}
\put(19,18.8){$y$}
\put(0.5,15){$\Vkt x_1$}
\put(98,15){$\Vkt x_2$}
\end{scriptsize}     
  \end{overpic} 
\end{center}	
\caption{Reduction of Stachel's double-Watt mechanism to a two-point guidance problem. 
}
  \label{fig:doubleStachel_reduced}
\end{figure}
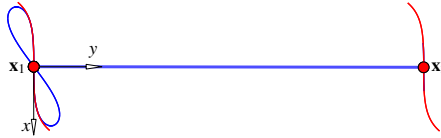

Then Stachel extended his double-Watt linkage by a Kempe-mechanism for the generation of the straight line motion of the midpoint $\Vkt x_3$ of $\Vkt x_1$ and $\Vkt x_2$ (see Fig.\ \ref{fig2}). 
In contrast, we only use a point guidance\footnote{This can also be interpreted in the terms of bar-joint framework, where the corresponding pin-joint is the ideal point of the $y$-axis.} to restrict the location of $\Vkt x_3$ on the line $y=\tfrac{3}{2}$ (cf.\ Fig.\ \ref{fig:doubleStachelextended}); i.e.\ 
\begin{equation}\label{eq:x3stachel}
        \Vkt x_3=\begin{pmatrix}
            \tau_3 \\
            \tfrac{3}{2}
        \end{pmatrix}. 
    \end{equation}
For a $n^{th}$-order flexibility at $t=0$ still 
Eq.\ (\ref{eq:F}) has to hold as well as the affine combination
\begin{equation}\label{eq:GH}
    \begin{pmatrix}
    G \\
    H
    \end{pmatrix}:=
    \Vkt x_1(\tau_1)+ \Vkt x_2(\tau_2) - 2  \Vkt x_3(\tau_3) = \Vkt o(t^{n}). %landau o
\end{equation}
We substitute Eq.\ (\ref{eq:taus}) into Eqs.\ (\ref{eq:F}) and 
(\ref{eq:GH}) and consider the coefficients $f_i$, $g_i$ and $h_i$ of $t^i$ in the resulting expressions.
It can easily be seen that $g_i=v_{1,i}+v_{2,i}-2v_{3,i}$ holds for all $i=1,2,\ldots$, thus we set 
\begin{equation}\label{eq3i}
    v_{3,i}=\tfrac{v_{1,i}+v_{2,i}}{2}.
\end{equation}
Moreover we get $f_1=h_1=h_2=0$ and $f_2=(v_{1,1}-v_{2,1})^2$. We set $v_{2,1}=v_{1,1}$ and obtain $f_3=0$ and $h_3=v_{1,1}^3$. 
Therefore this results in a $(1,2)$-flexibility. 

Now we consider the case $v_{1,1}=0$: Then we get $f_4=(v_{1,2}-v_{2,2})^2$. Thus we set $v_{2,2}=v_{1,2}$  and get $h_4=f_5=h_5=0$. Moreover we obtain $f_6=(v_{1,3}-v_{2,3})^2$, implying  $v_{2,3}=v_{1,3}$ and $h_6=v_{1,2}^3$. The latter shows a $(2,5)$-flexibility. 

Now we can set $v_{1,2}=0$ and proceed this procedure yielding the following sequence of flexion orders $(k,3k-1)$ for $k=1,2,\ldots$. 
According to Stachel the question remained open which is the correct order, as all the obtained ones are irreducible. This can be seen as follows:

\begin{figure}[t]
\begin{center}
\begin{overpic}
    [height=37mm]{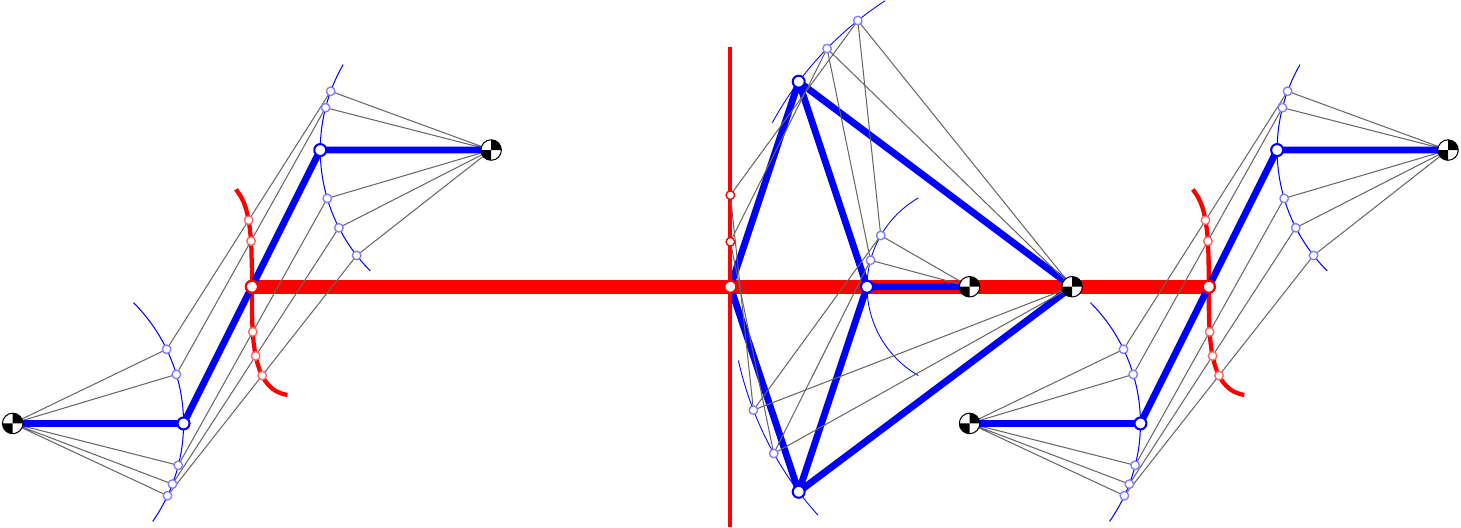}
  \end{overpic} 
\end{center}	
\caption{Stachel's double-Watt mechanism extended by a Kempe-mechanism (Figure by courtesy of Hellmuth Stachel \cite{stachel_aim}).}
  \label{fig2}
\end{figure}

\begin{figure}[t]
\begin{center}
\begin{overpic}
    [height=20mm]{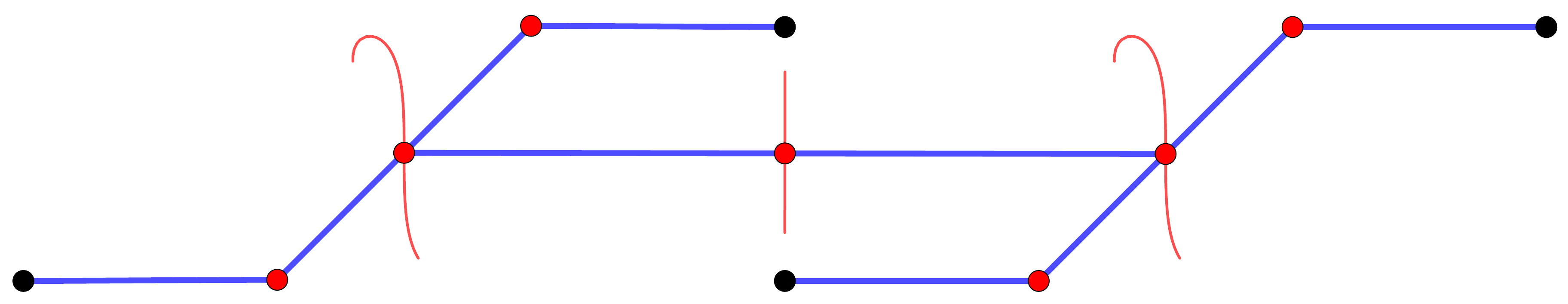}
    \begin{scriptsize}
\put(22.2,9){$\Vkt x_1$}
\put(75.7,9){$\Vkt x_2$}
\put(47,7){$\Vkt x_3$}
\end{scriptsize}     
  \end{overpic} 
\end{center}	
\caption{Stachel's double-Watt mechanism extended by the guidance of the midpoint $\Vkt x_3$ of $\Vkt x_1$ and $\Vkt x_2$ along a straight line.  
}
  \label{fig:doubleStachelextended}
\end{figure}

The conditions for a  $(1,2)$-flex which are 
\begin{equation}\label{eq:c12}
    v_{3,1}=\tfrac{v_{1,1}+v_{2,1}}{2}, \quad
    v_{3,2}=\tfrac{v_{1,2}+v_{2,2}}{2}, \quad
    v_{2,1}=v_{1,1}
\end{equation}
imply under the polynomial parameter substitution of Eq.\ 
(\ref{eq:polysubs}) a reducible $(p,3p-1)$-flexibility. 
Let's do this explicitly for $p=2$. Then we get:
\begin{align}
\overline{v}_{i,2}&=a_0v_{i,1} \\
\overline{v}_{i,3}&=a_1v_{i,1}  \\
\overline{v}_{i,4}&=a_0^2v_{i,2}+a_2v_{i,1} \\
\overline{v}_{i,5}&=2a_0a_1v_{i,2}+a_3v_{i,1}. 
\end{align}
Therefore the conditions for the $(2,5)$-flexibility, which are 
\begin{equation}
\overline{v}_{1,2}-\overline{v}_{2,2}=0, \quad
\overline{v}_{1,3}-\overline{v}_{2,3}=0, \quad
\overline{v}_{1,j}+\overline{v}_{2,j}-2\overline{v}_{3,j}=0
\end{equation}
for $j=2,3,4,5$ are fulfilled identically under 
Eq.\ (\ref{eq:c12}). But  $\overline{v}_{1,4}$ and $\overline{v}_{2,4}$ are in a certain relation as only $a_2$ can act as a free parameter, which in general has not to be the case. This shows the irreducibly of Stachel's $(2,5)$-flexibility. The same argument can be done also for the higher flexion orders in Stachel's sequence $(k,3k-1)$. 
\hfill $\diamond$
\end{example}

\begin{remark}\label{rem:proj}
Note that flexibility of $1^{st}$-order is invariant under projectivities \cite{wegner, wunderlich_shaky} but this does not hold for higher-orders (even not for affine transformations). \hfill $\diamond$
\end{remark}

\begin{remark}\label{rem:2}
Note that it is well known (cf.\ \cite[page 232]{sabitov} and \cite{alexandrov})  that there exists for each geometric structure an upper bound $n^*$ such that the 
$n^*$-order flexibility results in a continuous flexion. \hfill $\diamond$
\end{remark}

\subsection{Structures studied with respect to higher-order flexibility}\label{sec:studstruc}

Wohlhart \cite{wohlhart_degree}  followed Kuznetsov's approach (using a kinematic interpretation of the power-expansion in terms of 
velocity, acceleration, jerk, and so forth) for the study 
of higher-order flexible planar and spatial parallel manipulators of Stewart--Gough type. A deeper geometric study of these
planar mechanisms was done by Stachel in \cite{stachel_planar}, who also studied higher-order flexibility of bipartite planar frameworks \cite{stachel_bipartite}
as well as octahedra \cite{stachel_octahedra}. 
Open and closed spatial serial chains where studied in  \cite{chen,muller,wu}.
Kuznetsov \cite{kuznetsov} and Tarnai \cite{tarnai} demonstrate their theoretical considerations 
only on basis of some simple planar linkages, where the so-called Leonardo structure \cite{leonardo} 
has to be pointed out as in this way frameworks with a ($2^{\lambda}-1$)-order flex (according to Def.\ \ref{def1}) for 
arbitrary $\lambda\in\NN$ can be constructed (cf.\ Fig.\ \ref{fig3}). 
Local rigidity analysis of origami structures up to the $2^{nd}$-order were done by He and Guest \cite{he2021}. A characterization for $2^{nd}$-order flexibility of quad-surfaces with planar faces  was given by Schief et al.\ \cite{BHS08}. Finally, Tachi \cite{tachi_cap} capped rigid-foldable tubes with $2^{nd}$-order flexible structures. 

\begin{figure}[t]
\begin{center}
\begin{overpic}
    [height=25mm]{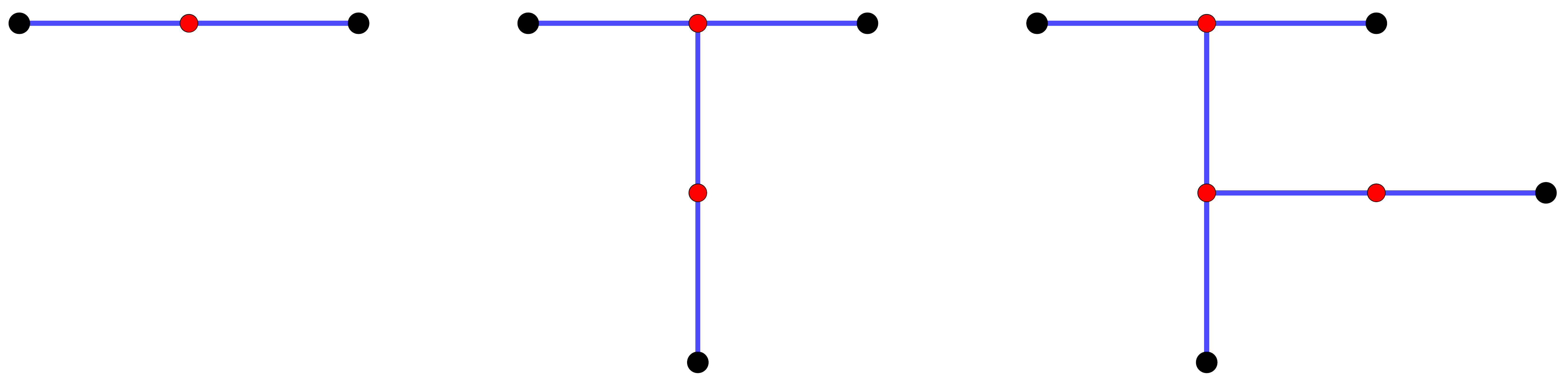}
\begin{scriptsize}
\put(0,20){$F_1$}
\put(11,20){$M_1$}
\put(22,20){$F_2$}
\put(33,20){$F_1$}
\put(45.2,20){$M_1$}
\put(45.5,11){$M_2$}
\put(46,0.5){$F_3$}
\put(54.5,20){$F_2$}
\put(65,20){$F_1$}
\put(77.7,20){$M_1$}
\put(87,20){$F_2$}
\put(77.7,10){$M_2$}
\put(86.7,9.5){$M_3$}
\put(97.5,9.5){$F_4$}
\put(78,0.5){$F_3$}
\end{scriptsize}     
  \end{overpic} 
\end{center}	
\caption{Leonardo structure for ${\lambda}=1$ (left), ${\lambda}=2$ (center) and ${\lambda}=3$ (right).}
  \label{fig3}
\end{figure}
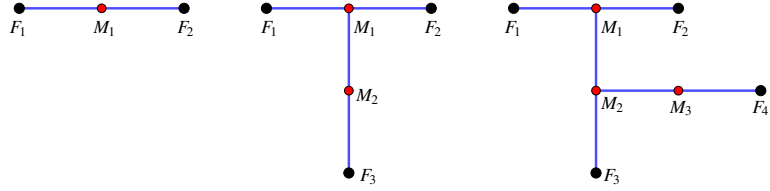   

\begin{remark}\label{rem:wunderlich}
    One should not forget about the work of Walter Wunderlich, who studied the geometry of several shaky structures and sometimes pointed out special ones with a higher-order flexibility (see overview article \cite{stachel_wunderlich}). \hfill $\diamond$
\end{remark}

Due to Remark \ref{rem:2} the idea of higher-order flexibility can also be used to compute over-constrained mechanisms. 
Based on the approach of Kuznetsov this method was stressed by Wohlhart \cite{wohlhart} to 
determine a special class of Stewart--Gough platforms with self-motions and by 
Bartkowiak and Woernle \cite{bart1,bart2} as well as Milenkovic \cite{milenkovic} for the design of overconstrained single-loop mechanisms. 
In contrast, Rameau and Serre \cite{rameau2015} focused on different computational methods of this problem. 
From the computation point of view also the work \cite{wampler} should be mentioned, where numeric algebraic geometry is used to test locally a
so-called {\it high-multiplicity infinitesimal degree of freedom} by means of Macaulay matrices.

\section{Redefinition of a higher-order flexibility and rigidity}\label{sec:redef}
According to M\"uller \cite{muller} the above mentioned examples with cusps in the configuration space (cusp mechanisms) demonstrate the basic shortcoming of any local mobility analysis using higher-order constraints.
Therefore we present a global approach, which is also inspired by an idea of Sabitov like Stachel's approach; namely by
 his finite algorithm for testing the bendability of a polyhedron \cite[page 231]{sabitov}. This  can be formulated as \smallskip follows:

\begin{figure}[t]
\begin{center}
\begin{overpic}
    [height=70mm]{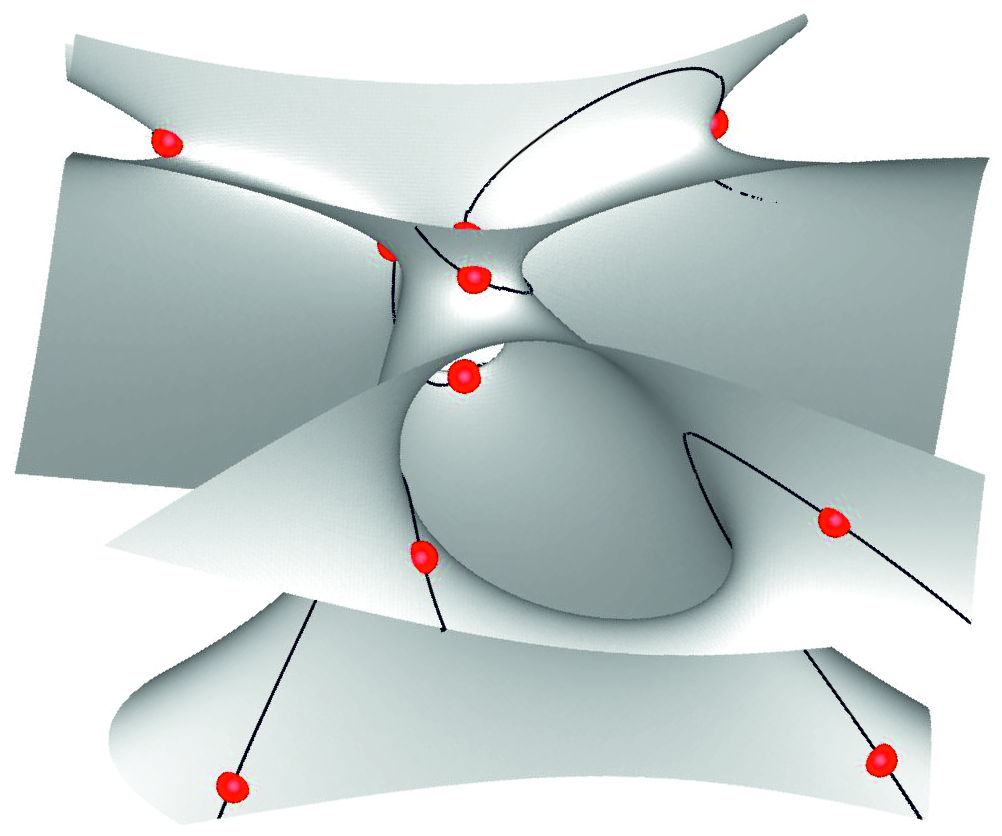}
  \end{overpic} 
\end{center}	
\caption{Illustration of the surface $\mathcal{S}_1$ (gray), the curve $\mathcal{S}_2$ (black) and the discrete set $\mathcal{S}_3$ of points (red) for the configurations of the 3-RPR manipulator discussed in Example \ref{ex:3rpr}.}
  \label{fig:3rpr}
\end{figure}

Let us consider the configuration-set $\mathcal{S}$ of all frameworks having the same connectivity but only differ in their 
intrinsic metric. 
Note that $\mathcal{S}$ is only a subset of 
 $\RR^{m}$ 
(due to the fact that edges are not allowed to have zero length). 
In the case of $1^{st}$-order flexibility each vertex $\Vkt x_i$ ($i=1,\ldots,w$) 
can be associated with a velocity vector $\Vkt x_{i,1}$ such that the edge lengths do not change instantaneously, where the 
set of velocity vectors is not allowed to originate from a rigid body motion (incl.\ standstill); i.e.\ no trivial $1^{st}$-order flex. 
The subset $\mathcal{S}_1\subset\mathcal{S}$ of $1^{st}$-order flexible configurations corresponds to 
the already mentioned shakiness variety $V_1$ in $\RR^{m}$. The sets $\mathcal{S}_j$ with $j>1$ are defined recursively as follows: 
If in a point of $\mathcal{S}_{j-1}$ 
a non-trivial $1^{st}$-order flex exists, which 
is tangential to 
$V_{j-1}$ then this point belongs to the set $\mathcal{S}_{j}$ 
thus we get a  hierarchical structure of 
flexibility of higher-order. 
A configuration is called  $n^{th}$-order flexible if it belongs to $\mathcal{S}_{n}$ but not to \smallskip $\mathcal{S}_{n+1}$.

\noindent
We proceed with a discussion of this approach:
\begin{enumerate}[$\bullet$]
    \item 
    This approach goes along with a recent result of Alexandrov \cite{alexandrov_new} for smooth surfaces, who was able to show that a $1^{st}$-order flex 
    tangential to $V_1$ can be extended to a $2^{nd}$-order flex.
     \item 
    Sabitov assumed that all the appearing sets $\mathcal{S},\mathcal{S}_1,\mathcal{S}_2,\ldots$ are manifolds and submanifolds, respectively. In general the varieties $V_1, V_2, \ldots$ contain singular points, which correspond mostly to the interesting configurations in the study of higher-order flexibility. 
    \item 
    An analogous assumption has to be done by Alexandrov \cite{alexandrov_new} in the smooth setting mentioned above, namely the restriction to regular points of $V_1$. 
\end{enumerate}

This means that this approach gives a proper definition of  $n^{th}$-order flexibility for configurations that correspond to points of $\RR^{m}$  which are regular with respect to each of the varieties $V_1, V_2, \ldots, V_n$.

\begin{lemma}\label{lem:1}
Every regular point of $V_1$  has to have a  single non-trivial instantaneous flexion. 
\end{lemma}

\begin{proof}
Let recall that $V_1$ is the zero set of the ideal generated by all minors $p_{1}, \ldots, p_{\mu}$ of $\Vkt R_{G(\Vkt X)}$ of order $m\times m$. 

Let $p_{j}$ equals the $\det(\Vkt r_1,\Vkt r_2, \ldots ,\Vkt r_m)$ where the $\Vkt r_i$'s denote pairwise distinct columns of the rigidity matrix $\Vkt R_{G(\Vkt X)}$ given in Eq.\ (\ref{eq:rigidityM}).

Now the entries of the gradient  of $p_{j}$, which is given by 
\begin{equation}
\nabla p_{j}=
\left( \tfrac{\partial p_{j}}{\partial z_1}, \tfrac{\partial p_{j}}{\partial z_2}, \ldots ,\tfrac{\partial p_{j}}{\partial z_m}\right),     
\end{equation}
can be computed due to the following product rule for determinants \cite[page 626]{adams}: 
\begin{equation}\label{detpart}
   \tfrac{\partial p_{j}}{\partial z_i}=
   \det(\tfrac{\partial \Vkt r_1}{\partial z_i},\Vkt r_2, \ldots, \Vkt r_m) + 
   \det(\Vkt r_1,\tfrac{\partial \Vkt r_2}{\partial z_i}, \ldots, \Vkt r_m) + \ldots +
   \det(\Vkt r_1,\Vkt r_2,\ldots,\tfrac{\partial \Vkt r_m}{\partial z_i}). 
\end{equation}
This already shows that for points 
 of $V_1$ with $rk(\Vkt R_{G(\Vkt X)})<m-1$ all these gradients $\nabla p_{1,j}$ are zero vectors, as all summands of Eq.\ (\ref{detpart})  are zero. As a consequence, these points have to be singular ones of 
  $V_1$. 
\hfill $\BewEnde$
\end{proof}

\begin{remark}
Lemma \ref{lem:1} explains Husty's observation for 3-RPR mechanisms given in \cite{husty2015}; namely ``{\it the surprising property that it} (singularity surface) {\it has a singularity itself at the point which corresponds to the pose with two dof local mobility}.''     

Moreover, this lemma also gives another reasoning for the note in \cite{steven} that the transverse rigidity test always fails if more than one non-trivial infinitesimal flex exist, 
as in this case the corresponding point on $V_1$ has to be a singular one. 
\hfill $\diamond$
\end{remark}

\begin{figure}[t]
\begin{center}
\begin{overpic}
    [height=27mm]{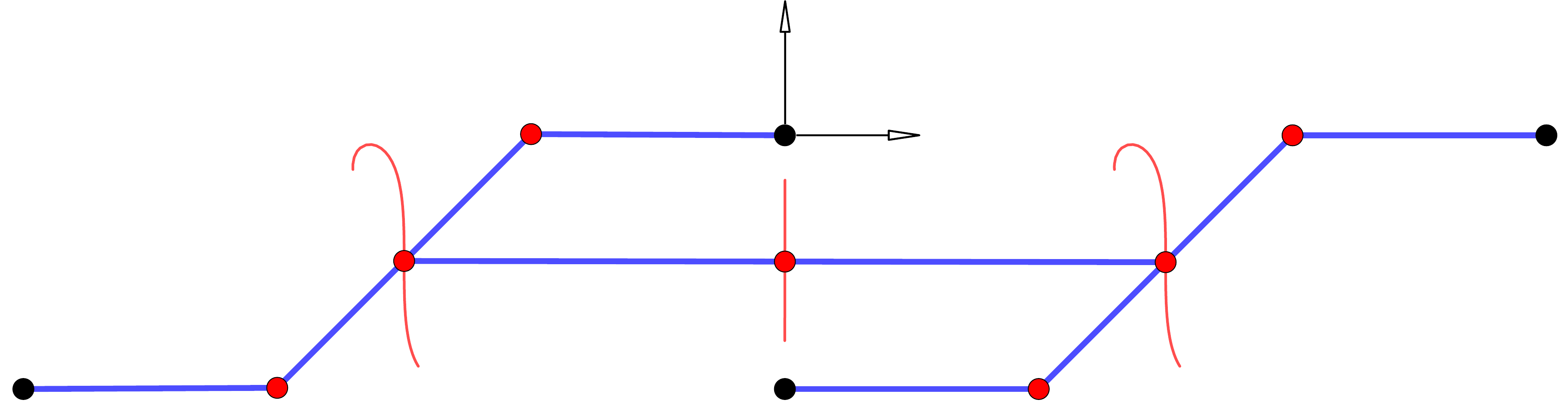}
    \begin{scriptsize}
    \put(57.5,18.7){$x$}  
    \put(51,25){$y$}  
\put(47,19){$F_0$}  
\put(32,19){$M_0$}  
\put(98,19){$F_2$}  
\put(81,19){$M_2$}  
\put(0.5,3){$F_1$} 
\put(15.5,3){$M_1$}
\put(46.8,1){$F_3$} 
\put(64,3){$M_3$}    
\put(21,9.5){$N_{0,1}$}
\put(75.7,9){$N_{2,3}$}
\end{scriptsize}     
  \end{overpic} 
\end{center}	
\caption{Stachel's double-Watt mechanism extended by the guidance of the midpoint $\Vkt x_3$ of $\Vkt x_1$ and $\Vkt x_2$ along a straight line.   
}
  \label{fig:doubleStachelextendedcoord}
\end{figure}
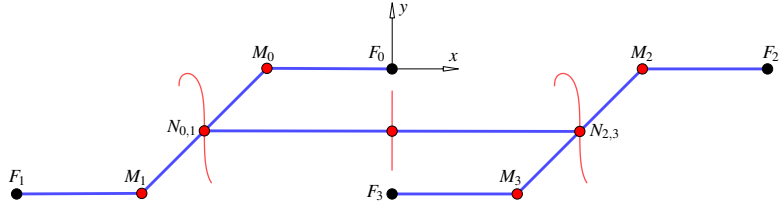    

\begin{example}\label{ex:watt4}
Let us consider Stachel's extended double-Watt framework introduced in Example \ref{ex:watt3}. For setting up the algebraic equations we use the following coordinatization according to Fig.\ \ref{fig:doubleStachelextendedcoord}: 
\begin{equation}
    F_0=(0,0)^T,  \quad
    F_1=(-3,-1)^T, \quad 
    F_2=(3,0)^T,  \quad
    F_3=(0,-1)^T,  
\end{equation}
for the points pinned to the base and 
\begin{equation}
    M_0=(a_0,b_0)^T,  \quad
    M_1=(a_1,b_1)^T, \quad 
    M_2=(a_2,b_2)^T,  \quad
    M_3=(a_3,b_3)^T,  
\end{equation}
for the moving points. Then the mechanism is determined by the following set of eight equations:
\begin{align}
&\|F_i-M_i\|^2=1 &\mathtext{for} &i=0,\ldots , 3 \label{eq:first}\\
&\|M_j-M_{j+1}\|^2=2 &\mathtext{for} &j=0,2\\ 
&\|N_{0,1}-N_{2,3}\|^2=9 &\phm \\
&a_0+a_1+a_2+a_3=0 \label{eq:last}
\end{align}
with $N_{0,1}=\tfrac{M_0+M_1}{2}$   and 
$N_{2,3}=\tfrac{M_2+M_3}{2}$. Note that Eq.\ (\ref{eq:last}) corresponds with 
the straight line motion of the midpoint of $N_{0,1}$ and 
$N_{2,3}$. 
Direct computations show that the rank of the $(8\times 8)$-rigidity matrix $\Vkt R_{G(\Vkt X)}$ in the configuration $\Vkt X$ given by
\begin{equation}\label{eq:cuspconf}
    (a_0,b_0,a_1,b_1,a_2,b_2,a_3,b_3)=(-1,0,-2,-1,2,0,1,-1)
\end{equation}
equals $6$. This confirms that $\Vkt X$ is a singular point of the variety $V_1$. \hfill $\diamond$
\end{example}

Beside Lemma \ref{lem:1} we also have to keep in mind that according to  
Remark \ref{rem:2} a certain flexibility order $n^*$ exists which implies flexibility of order $\infty$. 
With this ingredients  we can prove the following alternative characterization:

\begin{theorem}\label{thm1}
 The $n^{th}$-order flexibility with $n<n^*$ of a configuration which corresponds to a regular point of each variety $V_1, V_2, \ldots, V_n$
 is equivalent with the fact that the configuration is a framework realization of multiplicity $n+1$. 
\end{theorem}

\begin{proof}
    $V_1$ is the set of points determined by the constraint equations $c_1,\ldots ,c_l$ with multiplicity of at least two. In a general point $\Vkt X$ of $V_1$ the intersection multiplicity with respect to  $V_1\cap c_1\cap \ldots \cap c_l$ is 1. For increasing it a necessary and sufficient condition is that the tangent spaces have a positive-dimensional subspace in common. This is exactly the condition that in $\Vkt X$ an instantaneous flexion exists, which is tangential to  $V_1$. Due to Lemma \ref{lem:1} the common subspace has to be 1-dimensional.     Therefore by construction 
    a generic element of the variety $V_2$ has to have multiplicity 3 with respect to $V(c_1,\ldots ,c_l)$. 

    This line of argumentation can be iterated until we reach the set  $V_{n^*}$, which consists of points having multiplicity $\infty$. Thus points of $V_n\setminus V_{n+1}$ with $n<n^*$ have to correspond with framework realizations of multiplicity $n+1$. 
 \hfill $\BewEnde$
\end{proof}

A redefinition can be based on this property as it can also be extended to  singular points of the varieties $V_1, V_2, \ldots$ which are not covered by Sabitov's algorithm.

\begin{definition} \label{def3}
If a configuration does not belong to a continuous flexion of the framework then we define its order of flexibility by the number of coinciding framework realizations minus 1. 
\end{definition}

\begin{remark}
This definition follows the way Wunderlich (cf.\ Remark \ref{rem:wunderlich}) studied infinitesimal flexibility; namely as the limiting case where two realizations of a framework coincide (cf.\ Stachel \cite{stachel_between}). \hfill $\diamond$
\end{remark}

Based on Definition \ref{def3} we can also give a redefinition of higher-order rigidity as follows:

\begin{definition} \label{def4}
Is a configuration $n^{th}$-order flexible according to Definition \ref{def3} then it is $(n+1)$-rigid. 
\end{definition}

\subsection{Computational Aspects}\label{sec:algorithm}

As we now have obtained proper redefinitions of higher-order flexibility and rigidity, we remain with the problem of how  to compute the 
number of coinciding realizations.  
For that we have to calculate the intersection multiplicity of the hypersurfaces $c_1,\ldots ,c_l$ in the considered configuration $\Vkt X$. 
For the determination of the flexion order we suggest the following 3-step algorithm:

\begin{enumerate}
    \item 
    According to the Lasker–Noether theorem every algebraic set is the union of a finite number of uniquely defined algebraic sets known as irreducible components. They can be computed with an irredundant   primary decomposition\footnote{The prime decomposition is not valid as it does not preserve the intersection multiplicity.} algorithm (see e.g.\ \cite{zijia}). 
    \item 
    Then one has to test if the given realization is contained in a irreducible composition of dimension 1 or higher. If this is the case the configuration $\Vkt X$ is assigned with the flexion order $\infty$ (in accordance with Fulton's properties \cite{fulton} of an intersection number). If this is not the case then we identify all zero-dimensional primary ideals $I_1,\ldots, I_s$
    containing $\Vkt X$.
    \item 
    We compute the intersection multiplicity $q_i$ of $\Vkt X$ with respect to each primary ideal $I_i$ for $i=1,\ldots ,s$. Then the intersection multiplicity of $\Vkt X$ with respect to the hypersurfaces $c_1,\ldots ,c_l$ equals the sum $q_1+\ldots +q_s$.
\end{enumerate}

\begin{remark}\label{rem:unique}
According to Definition \ref{def3} the flexion order equals  $q_1+\ldots +q_s-1$, but if one is interested in a more detailed analysis of the configuration and its flexion order, then one should have a look at the sequence $(q_1,\ldots ,q_s)$. 
It is well known that the irredundant primary decomposition has not to be unique; but in our case we are save as we assumed that all primary ideals containing $\Vkt X$ are zero-dimensional. Therefore they have to correspond to minimal prime ideals and not to embedded ones, which are causing non-uniqueness (cf.\ \cite{zijia}). 
\hfill $\diamond$    
\end{remark}

In the following we sketch a possibility for the computation of $q_i$. Let us assume that the zero-dimensional primary ideal $I_i$ is generated by polynomials $g_1,\ldots, g_{\gamma}$. We distinguish the following two cases: 
\begin{enumerate}[a)]
\item
 If $\gamma=m$; i.e.\ $I_i$ is a complete intersection, then we can use theoretically the U-resultant method (see \cite[§\,18]{macaulay}, \cite[§\,83]{waerden} or \cite{kirby}), which works as follows:
One adds the so-called U-polynomial
\begin{equation}
    g_0=u_0+u_1z_1+\ldots +u_mz_m
\end{equation}
to the set $g_1,\ldots, g_m$ and eliminates $z_1,\ldots,z_m$
by means of Macaulay resultant \cite{macaulay_matrix}.
This results in a homogeneous polynomial $R(g_0,\ldots, g_m)$ where the degree equals the product of the degrees of $g_1,\ldots,g_m$. Moreover, $R(g_0,\ldots, g_m)$ factorizes into powers of $f$ linear factors 
\begin{equation}\label{eq:prod}
    \prod_{j=1}^f \left(\zeta_{j,0} u_0+\zeta_{j,1}u_1+\ldots +\zeta_{j,m} u_m \right)^{q_j}.
\end{equation}
Then the $j$th common point of $g_1,\ldots, g_m$ has multiplicity $q_j$ and its coordinates are given by $z_i=\zeta_{j,i}/\zeta_{j,0}$ for $i=1,\ldots,m$. 
\item 
If $\gamma>m$ one can use a generalization of the U-resultant method given by Lazard \cite{lazard} to end up with an expression of the form given in Eq.\ (\ref{eq:prod}).
\end{enumerate}
Let us demonstrate the above algorithm for the already mentioned Leonardo structure \cite{leonardo}.

\begin{example} \label{ex:tarnai}
According to Tarnai \cite{tarnai}  these frameworks with a ($2^{\lambda}-1$)-order flex can be generated by an iterative procedure. In the following we demonstrate this for ${\lambda}=1$, ${\lambda}=2$ and ${\lambda}=3$ (cf.\ Fig.\ \ref{fig3}), using the following coordinatization: 
\begin{equation}
    F_1=(-1,0)^T,\quad F_2=(1,0)^T,\quad 
F_3=(0,-2)^T,\quad F_4=(2,-1)^T,  
\end{equation}
for the points pinned to the base and 
\begin{equation}
    M_1=(a,b)^T,\quad  M_2=(c,d)^T,\quad  M_3=(e,f)^T,
\end{equation}
for the moving points.  
\begin{enumerate}[$\bullet$]
    \item 
    ${\lambda}=1$: In this case one has to solve the two equations $\|M_1-F_i\|^2=1$ for $i=1,2$, which read after homogenizing with $h$ as:
    \begin{equation}\label{eq:step1}
        a^2+2ah+b^2=0,\quad a^2-2ah+b^2=0.
    \end{equation}
    The primary decomposition of the ideal spanned by these two equations yields the two primary ideals $I_1^1=\langle a,b^2\rangle$ and  $I_2^1=\langle h,a^2+b^2\rangle$. Only $I_1^1$, which is zero-dimensional, contains the considered configuration $\Vkt X$ having homogeneous coordinates $(h:a:b)=(1:0:0)$. Computation of the U-resultant (with {\sc Macaulay2}) yields $u_0^2$, which shows that the configuration has multiplicity 2 and therefore a $1^{st}$-order flexion.
    \item
     ${\lambda}=2$: In addition to Eq.\ (\ref{eq:step1}) one has to consider the two conditions $\|M_2-M_1\|^2=1$ and $\|M_2-F_3\|^2=1$, which read after homogenizing with $h$ as:
     \begin{equation}\label{eq:step2}
     a^2 - 2ac + b^2 - 2bd + c^2 + d^2 - h^2=0, \quad
        c^2 + d^2 + 4dh + 3h^2=0.
    \end{equation}
    The primary decomposition (operated by {\sc Maple 2022}) of the ideal spanned by Eqs.\ (\ref{eq:step1}-\ref{eq:step2}) yields the following primary ideals
    \begin{equation}
    \begin{split}
     I_1^2=\langle &a, b^2, b - 2d - 2h, bh + c^2 \rangle, \\ 
     I_2^2=\langle &h, b^2 + a^2, d^2 + c^2,  bc-ad, bd + ca\rangle, \\
     I_3^2=\langle &a^3, h^2, ah, a^2b, hb, b^2 + a^2, ad - bc - 2ch, ac + bd + 2dh, c^2 + d^2-2ac - 2bd \rangle, \\
     I_4^2=\langle &a^4, c^5, h^6, ah, a^3c^4, a^3bc, hbc^2, b^2 + a^2, bh^3 + c^2h^2, 2cdh^2 -bch^2 + 2ch^3, \\ &2dh^3 -bh^3 + 2h^4,  c^2 + d^2 + 4dh + 3h^2, a^2d - abc - 2bh^2 + 4dh^2 + 4h^3, \\&ac + bd + 2dh + 2h^2, acd - bc^2 - 3bh^2 - 2c^2h + 2dh^2 + 2h^3\rangle.
     \end{split}
    \end{equation}
    Again only $I_1^2$, which is zero-dimensional, contains the considered configuration $\Vkt X$ having homogeneous coordinates $(h:a:b:c:d)=(1:0:0:0:-1)$. Computation of the U-resultant (with {\sc Macaulay2}) yields $2^4(u_0-u_4)^4$. This validates the $3^{rd}$-order flexion.
    \item 
    ${\lambda}=3$: In addition to Eqs.\ (\ref{eq:step1}) and (\ref{eq:step2}) one has to consider the two conditions $\|M_3-M_2\|^2=1$ and $\|M_3-F_4\|^2=1$, which read after homogenizing with $h$ as:
     \begin{equation}\label{eq:step3}
        c^2 - 2ce + d^2 - 2df + e^2 + f^2 - h^2, \quad
        e^2 - 4eh + f^2 + 2fh + 4h^2.
    \end{equation}
    The primary decomposition (operated by {\sc Maple 2022}) of the ideal spanned by Eqs.\ (\ref{eq:step1},\ref{eq:step2},\ref{eq:step3}) contains only\footnote{The other primary ideals with $h=0$ are not given due to their length.} one primary ideal with $h\neq 0$, which reads as:
    \begin{equation}\label{eq:wrongI}
        I_1^3=\langle a, b, c^2, c - 2e + 2h, 2e - c + 2d, e^2 + 2ef + f^2-2ce - cf   \rangle.
    \end{equation}
    But this cannot be correct as the U-resultant (with {\sc Macaulay2}) yields 
    $2^8(u_0-u_4+u_5-u_6)^4$, which shows only a 4-fold realization at the considered configuration $\Vkt X$ having homogeneous coordinates $(h:a:b:c:d:e:f)=(1:0:0:0:-1:1:-1)$. 
    
    We did a recheck following the idea of \cite{weil} by slightly perturbating  the system of equations. Then it can easily be seen that there are 8 solutions\footnote{This number can additionally be verified by the {\tt IntersectionMultiplicity} command implemented in {\sc Maple 2022}.} in the neighborhood of $\Vkt X$. 
    
     This shows up a problem of the {\tt PrimaryDecomposition} command in {\sc Maple 2022}.     
    In order to correct $I_1^3$ of Eq.\ (\ref{eq:wrongI}) one has to replace $b$ by $b^2$ (as this is the case in $I_1^1$ and $I_1^2$). Then the U-resultant (operated with {\sc Macaulay2}) yields the expected expression $2^{16}(u_0-u_4+u_5-u_6)^8$.  \hfill $\diamond$
\end{enumerate}
\end{example}

\begin{example}\label{ex:watt5}
Continuation of Example \ref{ex:watt4}:
The primary decomposition (operated by {\sc Maple 2022}) of the ideal spanned by Eq.\ (\ref{eq:first}--\ref{eq:last}) yields only one zero-dimensional primary ideal $I$ containing the configuration $\Vkt X$ of Eq.\ (\ref{eq:cuspconf}); namely
\begin{equation}\label{eq:idealI}
\begin{split}
 I=\langle
&(1 + a_0)^2, (a_2 - 2)^2, (3 - a_2)^2 + b_2^2 - 1, (a_0 + 3)a_1 + 5 + (b_0 + 1)b_1, \\
&(a_3 - 3)a_2 + 5 + (b_2 + 1)b_3, a_1 + a_0 + a_3 + a_2, a_0^2 + b_0^2 - 1, \\ 
& (6-2a_3)a_2 + a_3^2 - 2b_2b_3 + b_3^2 - 10,   a_1^2-2a_0a_1 - 2b_0b_1 + b_1^2 - 1, \\
&(a_0 - a_2 - a_3 - 3)a_1 + (3-a_0 + a_3)a_2 + (b_0 - b_2 - 
b_3 - 1)b_1 + \\
&(b_2-b_0  - 1)b_3 - a_0a_3 - b_0b_2 - 26
 \rangle .      
\end{split}
\end{equation}
As this ideal has more than eight generators, we cannot apply the U-resultant method as done in Example \ref{ex:tarnai}. 
As we are not aware of any implementation of the generalized U-resultant method of Lazard \cite{lazard}, we proceeded as follows: The ideal $I$ of Eq.\ (\ref{eq:idealI}) only has the 
solution $\Vkt X$ and we determined its multiplicity by the {\sc  Maple} command {\tt NumberOfSolutions}\footnote{In {\sc Maple 2022} there is no documentation on how the command {\tt NumberOfSolutions} works.}
of the {\tt PolynomialIdeals} package, which yields $6$. 

As one cannot trust for sure the {\tt PrimaryDecomposition} command in {\sc Maple 2022} as demonstrated in Example \ref{ex:tarnai}, we did again a recheck by the pertubation approach of \cite{weil}, which confirms multiplicity 6.
According to Definition \ref{def3} this implies a flexion of order 5.  
\hfill $\diamond$
\end{example} 

\begin{example}\label{ex:watt6}
We can also force the midpoint $\Vkt x_3$ of $\Vkt x_1$ and $\Vkt x_2$ of the original double-Watt mechanism of Connelly and Servatius (cf.\ Example \ref{ex:watt1}) to run on a vertical line. In analogy to Example \ref{ex:watt3} one can use Stachel's approach, which yields the sequence of flexion orders $(k,3k-1)$ for odd $k$ and $(k,3k+\tfrac{k}{2}-1)$ for even $k$.

Similar considerations as in Example \ref{ex:watt5} show, that in this case seven solutions\footnote{A slight perturbation of the system of equations shows that seven solutions converge against the given configuration. Note that in this case the {\tt PrimaryDecomposition} command in {\sc Maple 2022} does not work as the resulting solution is only sixfold and also the {\tt IntersectionMultiplicity} command fails for all possible $8!=40320$ permutations.} coincide, yielding flexion order 6. 
\hfill $\diamond$
\end{example} 

\begin{remark}
    The flexion order of the frameworks discussed in Examples \ref{ex:watt5} and \ref{ex:watt6} can be raised from  5 and 6  to 9 and 14, respectively, by modifying the dimension of the used Watt-linkage in a way that the coupler is vertical (i.e.\ the coupler is tangential to the considered branch) but the arms remain horizontal in the considered configuration.   
	\hfill $\diamond$
\end{remark}

\begin{figure}[t]
\begin{center}
\begin{overpic}
    [height=20mm]{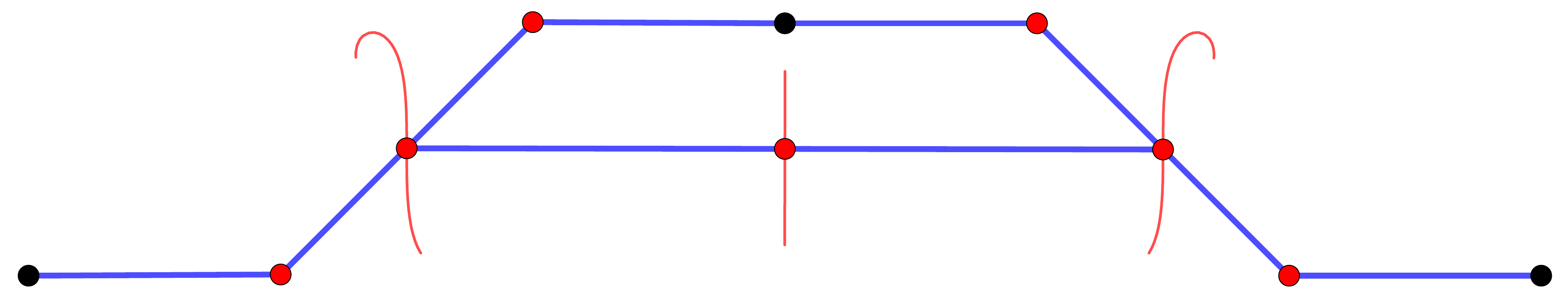}
\begin{scriptsize}
\put(22.6,10){$\Vkt x_1$}
\put(75.5,10){$\Vkt x_2$}
\put(47.5,7.8){$\Vkt x_3$}
\end{scriptsize}        
  \end{overpic} 
\end{center}	
\caption{Double-Watt mechanism of Connelly and Servatius extended by the guidance of the midpoint $\Vkt x_3$ of $\Vkt x_1$ and $\Vkt x_2$ along a straight line.  
}
  \label{fig:doubleoriginalextended}
\end{figure}

Clearly as the algorithm given in Section \ref{sec:algorithm} is based on symbolic methods from computer algebra, we are faced with computational limits. But beside this problem the flexion order of a given configuration can be computed with the presented tools in all cases from the pure theoretical point of view. 
Another problem is the computation of configurations with a higher-order flexion, which is discussed in the next section.

\section{Computing 3-RPR configurations with a higher-order flexion}\label{sec:computing}

In this section we demonstrate, how the idea of Sabitov's finite algorithm for testing the bendability of a polyhedron \cite[page 231]{sabitov} can be used to compute iteratively configurations with a higher-order flexion. 
We do this exemplarily for a planar 3-RPR manipulator consisting of a moving triangle which is connected by three legs to the fixed base. The legs are jointed to the platform and the base by rotational (R) joints and the corresponding anchor points are denoted by $m_i$ and $M_i$, respectively, for $i=1,2,3$. The length $r_i$ of the legs can actively be controlled by prismatic (P) joints.

Our choice of the example was motivated by the following statement of Husty \cite{husty2023} that
$3^{rd}$-order flexibility ``{\it can be reached by any design because the three necessary conditions could be imposed on the input parameters only. Unfortunately neither the conditions nor the number of corresponding poses are known}''. We will clarify this in Section \ref{sec:barplate}. 

 Note that we can interpret the triangular base and platform either as (a) triangular plates or (b) triangular bar structures. In case (a) the 3-RPR manipulator can be seen as a pin-jointed bar-plate framework and in case (b) as a classical bar-joint framework. In the following Subsections \ref{sec:barplate} and \ref{sec:bar}  we distinguish these two interpretations as they will effect the discussion of configurations with a higher-order flexion. But let us start with some review on this topic.

As already mentioned in Section  \ref{sec:studstruc}
Wohlhart \cite{wohlhart_degree}  followed a kinematic version of Kuznetsov's approach for the study of higher-order flexible  3-RPRs (interpreted as bar-plate frameworks).  Stachel studied the geometry of higher-order flexible 3-RPRs (interpreted as bar-joint frameworks) in \cite{stachel_planar}, where he has shown the following result for a configuration of flexion order $(1,n)$:

{\it If one disconnects the leg $M_im_i$ from the platform, then the trajectory of the point $m_i$ under the resulting four bar motion has $n^{th}$-order contact with the circle centered in $M_i$ having radius $r_i$.}\footnote{According to  \cite[Lem.\ 1]{stachel_planar} a corresponding result also holds for Stewart--Gough platforms, which goes along with the definition of an ``{\it order of a configuration}'' given by Sarkissyan and Parikyan \cite{parikyan} in 1990 (see also Wohlhart \cite[page 1116]{wohlhart_degree}).} 

Moreover, this result implies that in this configuration $(n+1)$ realizations coincide, which also goes along with our redefinition given in Definition \ref{def3}. Based on this characterization Husty \cite{husty2023} has given an approach for the computation of 3-RPR configurations (interpreted as bar-plate frameworks) with flexion order 5, which has to be done carefully as it can also yield pseudo-solutions\footnote{Note that the example illustrated in Fig.\ 8 of \cite{husty2023} does not show a $5^{th}$-order flexion, as it is not 
a sixfold solution of the direct kinematics problem. The direct kinematic splits up into a fourfold solution and a twofold one. Therefore the two corresponding configurations are flexible of order 3 and 1, respectively.}.

\subsection{Bar-plate framework}\label{sec:barplate}

Let us start with the computation of $V_1$ for these mechanisms, which can be done in several ways. For the problem at hand we stress an approach of Husty and Gosselin \cite{husty_gosselin}, which is recapped next:

The coordinates $(a_i,b_i)^T$ of a point $m_i$ of the moving platform with respect to the moving frame can be transformed into coordinates of the
fixed frame using the so-called Blaschke-Gr\"unwald parameters $(q_0:q_1:q_2:q_3)$. They can be seen as homogeneous coordinates of points of a projective 3-dimensional space $P^3$. It is well known, that there is a bijection between points of this space sliced along the line $q_0=q_1=0$ and the planar motion group SE(2). 
The slicing has to be done to ensure that the 4-tuple $(q_0:q_1:q_2:q_3)$ can be normalized by $c_4=0$ with
\begin{equation}
c_4:=    q_0^2+q_1^2-1.
\end{equation}
If this normalization condition holds the above mentioned transformation reads as follows:
\begin{equation}
    \begin{pmatrix}
        a \\ b
    \end{pmatrix}
    \mapsto 
    \begin{pmatrix}
        q_0^2-q_1^2 & -2q_0q_1 \\
        2q_0q_1 &  q_0^2-q_1^2 
    \end{pmatrix}
    \begin{pmatrix}
        a \\ b
    \end{pmatrix}+
    \begin{pmatrix}
        2q_1q_2+2q_0q_3 \\ 2q_1q_3-2q_0q_3
    \end{pmatrix}.
\end{equation}
Using these Blaschke-Gr\"unwald parameters the condition that a point $m_i$ is located on a circle with radius $r_i$ around the fixed point $M_i$ with coordinates $(A_i,B_i)^T$ with respect to the fixed frame, can be written as $c_i=0$ with: 
\begin{equation}
\begin{split}
c_i:= & 2A_ia_iq_1^2 - 2A_ia_iq_0^2 + 4A_ib_iq_0q_1 - 4B_ia_iq_0q_1 - 2B_ib_iq_0^2 + 2B_ib_iq_1^2 + \\
&a_i^2q_0^2 + a_i^2q_1^2 + b_i^2q_0^2 + b_i^2q_1^2 - 4A_iq_0q_3 - 4A_iq_1q_2 + 4B_iq_0q_2 - 4B_iq_1q_3 + \\
&4a_iq_0q_3 - 4a_iq_1q_2 - 4b_iq_0q_2 - 4b_iq_1q_3 + A_i^2 + B_i^2 + 4q_2^2 + 4q_3^2 - r_i^2
\end{split}
\end{equation}
The information of the leg lengths $r_i$ complete the intrinsic metric of the framework. Then its realizations\footnote{In this context the realizations are also known as solutions of the direct kinematics problem.}  $G(\Vkt X)$ are obtained as the solutions of the four algebraic equations $c_1=c_2=c_3=c_4=0$. 
It is well-known that there can only exist six solutions thus a 
$6^{th}$-order flex (according to Definition \ref{def3}) implies a continuous flexion; i.e.\ $n^*=6$. 

Now we are looking for poses of the platform yielding an infinitesimal flexibility of the framework.
As described in Section \ref{sec:algapproach}, these configurations are characterized by the fact that the determinant of the  rigidity matrix $\Vkt R_{G(\Vkt X)}$ vanishes, which is given by
\begin{equation}
\Vkt R_{G(\Vkt X)}=(\nabla c_1,\nabla c_2,\nabla c_3,\nabla c_4)=
\begin{pmatrix}
    \frac{\partial c_1}{\partial q_0} &   \frac{\partial c_2}{\partial q_0} & \frac{\partial c_3}{\partial q_0} &  \frac{\partial c_4}{\partial q_0} \\
    \frac{\partial c_1}{\partial q_1} &   \frac{\partial c_2}{\partial q_1} & \frac{\partial c_3}{\partial q_1} &  \frac{\partial c_4}{\partial q_1} \\
    \frac{\partial c_1}{\partial q_2} &   \frac{\partial c_2}{\partial q_2} & \frac{\partial c_3}{\partial q_2} &  \frac{\partial c_4}{\partial q_2} \\
    \frac{\partial c_1}{\partial q_3} &   \frac{\partial c_2}{\partial q_3} & \frac{\partial c_3}{\partial q_3} &  \frac{\partial c_4}{\partial q_3} 
\end{pmatrix}    
\end{equation}
according to Eq.\ (\ref{eq:rigidityM}). 
Then the shakiness variety $V_1$ equals the zero set of $s:=\det\left( \Vkt R_{G(\Vkt X)} \right)$. 
According to \cite{adit} this variety has only singularities for some special designs beside the singularities resulting from the parametrization, which equal the line $q_0=q_1=0$.
Therefore in the generic case each point of $V_1$ sliced along the line $q_0=q_1=0$ is a regular one. Thus according to Lemma \ref{lem:1} the tangent planes to $c_1,\ldots ,c_4$ have a line in common. The orthogonality of this line to $\nabla s$ is equivalent to the condition 
\begin{equation}
    rk(\nabla c_1,\nabla c_2,\nabla c_3,\nabla c_4,\nabla s)=3
\end{equation}
which implies the four conditions $s_1=s_2=s_3=s_4=0$ with:
\begin{align}
s_1:=&\det(\nabla c_2,\nabla c_3,\nabla c_4,\nabla s),    &\quad
s_2:=&\det(\nabla c_1,\nabla c_3,\nabla c_4,\nabla s),    \\
s_3:=&\det(\nabla c_1,\nabla c_2,\nabla c_4,\nabla s),    &\quad
s_4:=&\det(\nabla c_1,\nabla c_2,\nabla c_3,\nabla s).    
\end{align}
Then $V_2$ is the zero set of the ideal 
\begin{equation}\label{eqI2}
 I_2 =   \langle  
    s, s_1, s_2, s_3, s_4
    \rangle.
\end{equation}
Iteration of the above procedure yields the conditions $s_{1,i}=s_{2,i}=s_{3,i}=s_{4,i}=0$ with:
\begin{align}
s_{1,i}:=&\det(\nabla c_2,\nabla c_3,\nabla c_4,\nabla s_i),    &\quad
s_{2,i}:=&\det(\nabla c_1,\nabla c_3,\nabla c_4,\nabla s_i),    \\
s_{3,i}:=&\det(\nabla c_1,\nabla c_2,\nabla c_4,\nabla s_i),    &\quad
s_{4,i}:=&\det(\nabla c_1,\nabla c_2,\nabla c_3,\nabla s_i),    
\end{align}
for $i=1,\ldots, 4$. 
Then $V_3$ is the zero set of the ideal 
\begin{equation}\label{eqI3}
I_3=    \langle  
    s, s_1, s_2, s_3, s_4,s_{1,1},\ldots,s_{4,1}, 
    s_{1,2},\ldots,s_{4,2},
    s_{1,3},\ldots,s_{4,3},
    s_{1,4},\ldots,s_{4,4}
    \rangle.
\end{equation}
In addition the singular points of $V_2$ have to considered separately. As $V_2$ is a curve in $P^3$ a singularity corresponds to the case 
\begin{equation}\label{eq:singcurvpoints}
    rk(\nabla s,\nabla s_1,\nabla s_2,\nabla s_3,\nabla s_4)=1. 
\end{equation}
In the following we apply this procedure to a concrete example.

\begin{example}\label{ex:3rpr}
The geometry of the platform and base is given by:
\begin{equation}\label{eq:geomRPR}
\begin{split}
    &A_1=0,\quad B_1=0,\quad A_2=3,\quad B_2=0, \quad A_3=1,\quad B_3=3, \\
    &a_1=0,\quad b_1=0,\quad a_2=1,\quad b_2=0, \quad a_3=2,\quad b_3=1. 
\end{split}    
\end{equation}
For this values we obtain 
\begin{equation}
\begin{split}
     s =\, &5q_0^3q_2 - 13q_0^2q_1q_2 - 4q_0^2q_1q_3 + 5q_0^2q_2^2 + 7q_0q_1^2q_2 + 11q_0q_1^2q_3 - \\
    &6q_0q_1q_2^2 - 6q_0q_1q_3^2 + 10q_1^3q_3 - 5q_1^2q_3^2.
\end{split}
\end{equation}
In the next step we consider the ideal $I_2$ given in Eq.\ (\ref{eqI2}). By means of Hilbert dimension it can be verified that $V_2$ is a curve in $P^3$. Moreover, the degree of $V_2$ is 18. But $V_2$ splits up into a curve $g$ of degree 14 and the line $q_0=q_1=0$ of multiplicity 4, which can be seen as follows: We add the expression $(q_0^2+q_1^2)u-1$ to the ideal $I_2$ and eliminate the unknown $u$ (also known as the Rabinowitsch trick). The variety of the resulting elimination ideal is only of degree 14. Moreover, it can be checked that only $(q_0:q_1)=(0:0)$ fulfill the equations 
$s=s_1=s_2=s_3=s_4=0$ but not $(q_0:q_1)=(1:\pm I)$.  
Due to the slicing of $P^3$ along $q_0=q_1=0$ we can restrict to the curve $g$ of degree 14. It can easily be checked that $g$ does not contain any singular points by applying the criterion of Eq.\ (\ref{eq:singcurvpoints}). 

In the last step we consider the ideal  $I_3$ given in Eq.\ (\ref{eqI3}).  $V_3$ again contains the line $q_0=q_1=0$ with  multiplicity 3. We can get rid of this line in the same way as done in the case of $V_2$ (Rabinowitsch trick). The elimination ideal then yields 32 solutions. We can even eliminate $q_0$ and $q_3$ from the set of equations generating $I_3$ to end up with the polynomial of degree 32, which is given in the Appendix. 
By setting $q_1=1$ we can easily check that it has 10 real solutions, which are given in Table \ref{tab1}. Moreover, the corresponding configurations are illustrated in Fig.\ \ref{fig4}. 
In addition, the shakiness variety $V_1$, the curve $g$ and the 10 configurations are illustrated in Fig.\ \ref{fig:3rpr} for $q_0=1$.

Finally it should be noted, that it remains unclear 
if examples with 32 real solutions exist. \hfill $\diamond$
\end{example}

\begin{table}[t]
\caption{All real RPR-configurations with a $3^{rd}$-order flex for the geometry given in Eq.\ (\ref{eq:geomRPR})}
\centering
\begin{tabular}{|c||c|c|c|c|c|}
\hline
\# & $q_0$ & $q_1$ & $q_2$ & $q_3$ & Fig.\ \ref{fig4}\ \\ \hline
1 &
 $\phm 0.612011087187$ &
 $0.790849182309$ &
$-1.605503824990$ &
 $\phm 0.608460800603$ &
(a)
\\ \hline
2 &
$-0.335887854729$ &
$0.941901984839$ &
$-0.360610831902$ &
$\phm 2.136077449950$ &
(b)
\\ \hline
3  &
 $\phm 0.933493296982$ &
 $0.358594847269$  &
$-0.518343596625$  &
 $\phm 0.387989956333$ &
(c)
\\ \hline
4 &
$-0.351833124675$ &
 $0.936062739553$ &
 $\phm 1.595410958064$ &
 $\phm 1.897762719666$ &
(d)
\\ \hline
5 &
 $\phm 0.926572314644$ &
 $0.376116665058$ & 
 $\phm 0.064697675622$ &
$-0.063224689351$ &
(e)
\\ \hline
6 &
$-0.985793710397$ &
 $0.167960592226$ &
 $\phm 0.665253728293$ &
 $\phm 2.206010002417$ &
(f)
\\ \hline
7 &
$-0.388425191626$ &
 $0.921480260510$ &
 $\phm 0.083061189759$ &
 $\phm 0.100978528116$ &
(g)
\\ \hline
8 &
$-0.430899664574$ &
 $0.902399844342$ &
 $\phm 1.635384670001$ &
$-0.158191823892$ &
(h)
\\ \hline
9 &
$\phm 0.700957636960$ &
$0.713202910248$ &
$\phm 1.476082504043$ &
$\phm 0.619974829761$ &
(i)
\\ \hline
10 &
$-0.981604898439$ &
 $0.190923606082$ &
 $\phm 0.557314730844$ &
$-1.086046127580$ &
(j)
\\ \hline
\end{tabular}
\label{tab1}
\end{table}

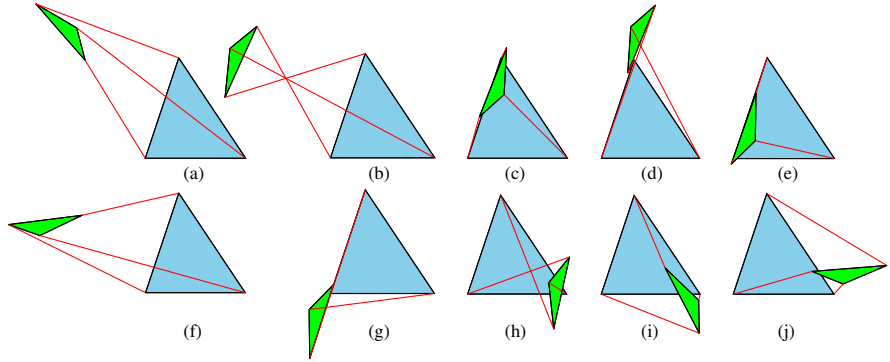
\begin{figure}[t]
\begin{center}
\begin{overpic}
    [width=\textwidth]{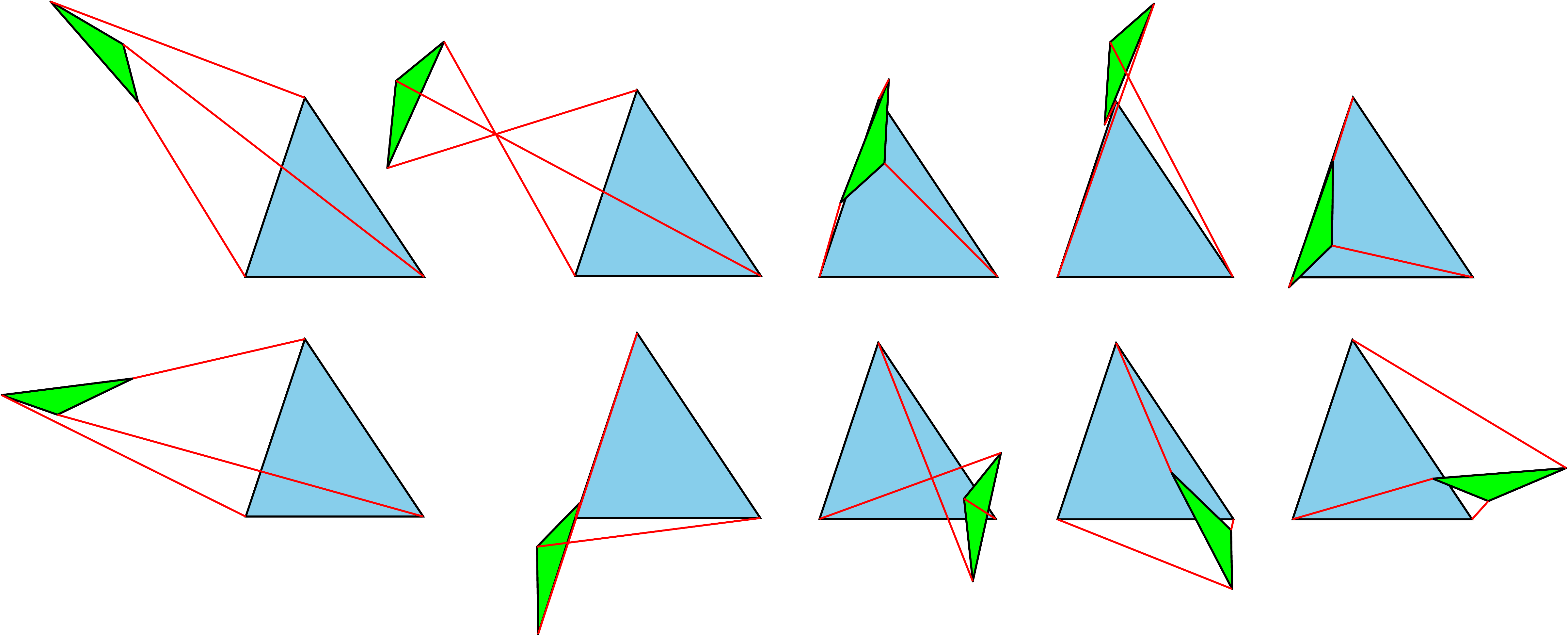}  
    \begin{scriptsize}
\put(20,20.5){(a)}
\put(41,20.5){(b)}
\put(56.5,20.5){(c)}
\put(72,20.5){(d)}
\put(87.5,20.5){(e)}
\put(20,2.5){(f)}
\put(41,2.5){(g)}
\put(56.5,2.5){(h)}
\put(72,2.5){(i)}
\put(87.5,2.5){(j)}
\end{scriptsize}  
  \end{overpic} 
\end{center}	
\caption{Visualization of the 10 configuration with a $3^{rd}$-order flex given in Table \ref{tab1}.}
  \label{fig4}
\end{figure}

\subsection{Bar-joint framework}\label{sec:bar}

For the interpretation as bar-joint framework there exists 24 realizations, as the platform triangle as well as the base triangle can flip. But this does not imply that $n^*=24$ holds true for all cases, as the following study will show. 

As we assumed in Section \ref{sec:intro} that bar lengths are always non-zero we can assume a rescaling of the framework such that the bar between $M_1$ and $M_2$ has length one. Then the pin-joints can be coordinatized as follows with respect to the fixed frame:
\begin{equation}
   M_1=(0,0)^T, \quad  M_2=(1,0)^T, \quad  M_3=(A_3,B_3)^T, \quad  
   m_j=(a_j,b_j)^T,
\end{equation}
for $j=1,2,3$. If the remaining 8 bar lengths are known they imply 8 distance equations $c_1,\ldots, c_8$. The solutions of this set of equations correspond to realizations of this isostatic bar-joint framework. Then we can compute the $(8\times 8)$ rigidity matrix according to 
Eq.\ (\ref{eq:rigidityM}). 
Again the shakiness variety $V_1$ is given as the zero set of $\det\left( \Vkt R_{G(\Vkt X)} \right)$ which splits up into the following three factors $s_1s_2s_3$ with:
\begin{equation}
\begin{split}
    s_1=&B_3, \\
    s_2=&a_1b_2 - a_1b_3 - a_2b_1 + a_2b_3 + a_3b_1 - a_3b_2, \\
    s_3=&A_3a_1b_2b_3 - A_3a_2b_1b_3 - B_3a_1a_3b_2 + B_3a_2a_3b_1 - \\
    &A_3b_1b_2 + A_3b_1b_3 + B_3a_1b_2 - B_3a_3b_1 - a_1b_2b_3 + a_3b_1b_2
\end{split}
\end{equation}
Their geometric interpretation is that for $s_1=0$ (resp.\ $s_2=0$) the base (resp.\ platform) degenerates into a line\footnote{A triangle and its mirrored version can only coincide if it degenerates into a line.}. For $s_3=0$ the three legs belong to a pencil of lines. 
Let us denote the varieties $s_i=0$ by $S_i$ for $i=1,2,3$. 
Now we can easily identify the following regions of $V_1$ where different values for $n^*$ hold true:
\begin{align}
&S_1\setminus (S_2 \cup S_3) &\quad & n^*=2  \label{eq:zwei1}\\ 
&S_2\setminus (S_1 \cup S_3) &\quad & n^*=2 \label{eq:zwei2} \\
&S_3\setminus (S_1 \cup S_2) &\quad & n^*=6 \\
&(S_1\cap S_2)\setminus S_3  &\quad & n^*=4 \label{eq:vier}\\
&(S_1\cap S_3)\setminus S_2  &\quad & n^*=12 \\
&(S_2\cap S_3)\setminus S_1  &\quad & n^*=12 \\
&S_1\cap S_2 \cap S_3  &\quad & n^*=24
\end{align}
\vspace{-5mm}
\noindent
\begin{remark}
We are aware of the fact that no point on $S_1\setminus (S_2 \cup S_3)$ or $S_2\setminus (S_1 \cup S_3)$ can reach a higher flexion order than 1, as a triangle does not allow an isometric deformation according to the side-side-side theorem. Therefore $n^*=2$ of Eqs.\ (\ref{eq:zwei1}) and (\ref{eq:zwei2}) as well as $n^*=4$ of Eq.\ (\ref{eq:vier}) can never be reached and are only of theoretical nature. \hfill $\diamond$    
\end{remark}

In the following we give the construction of configurations with the highest possible flexion order. Let us assume that the platform and the base triangles degenerate into lines $l$ and $L$, respectively. 
A necessary condition for a configuration of flexion order 23, is that $l$ and $L$ coincide. If this would not be the case one can reflect the configuration on one of these lines to get another realization, which contradicts the assumption that all 24 realization coincide. 

\begin{remark}
Interestingly such a configuration is not only a singular point of $V_1$, as it is located in the intersection of $S_1$, $S_2$ and $S_3$ but already a singular point of $S_3$ according to \cite{adit}. \hfill $\diamond$    
\end{remark}

Therefore  a $23^{rd}$-order flexible bar-joint framework follows from a $5^{th}$-order flexible plate-bar framework, where all six anchor points are located on a line. This problem can be solved following the already mentioned approach of Husty  \cite{husty2023}. In this way the following example was computed.

\begin{example}\label{ex:flex23}
    The geometry of the base is given by
    \begin{equation}
        M_1=(0,0)^T,\quad 
        M_2=(1,0)^T, \quad
        M_3=(5,0)^T
    \end{equation}
    with respect to the fixed system and the geometry of the platform is given by 
    \begin{equation}
        m_1=(0,0)^T, \quad
        m_2=(\tfrac{1}{2} + \tfrac{2\sqrt{10}}{5} - \tfrac{\sqrt{120\sqrt{10}-255}}{10},0)^T, \quad
        m_3=(3,0)^T
    \end{equation}
    with respect to the moving frame. The information on the intrinsic metric of the framework is completed by the following lengths of the three legs:
    \begin{equation}
        r_1=\tfrac{3}{2} + \tfrac{2\sqrt{10}}{5} - \tfrac{\sqrt{120\sqrt{10}-255}}{10}, \quad
        r_2=2, \quad
        r_3=\tfrac{7}{2} + \tfrac{2\sqrt{10}}{5} - \tfrac{\sqrt{10}\sqrt{48\sqrt{10}-102}}{20}.
    \end{equation}
This configuration is illustrated in Fig.\ \ref{fig5} where also pictures of a model can be seen, which was produced for validation of the higher-order flexion. 
\hfill $\diamond$
\end{example}

\begin{remark}
    It should be possible to determine the set  of these frameworks with flexion order 23 in full generality (as only 4 unknowns are involved), 
		which is dedicated to future research. \hfill $\diamond$
\end{remark}

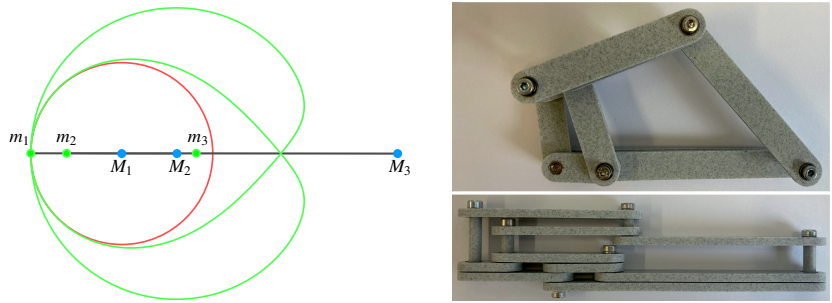
\begin{figure}[t]
\begin{center}
\begin{overpic}
    [width=50mm]{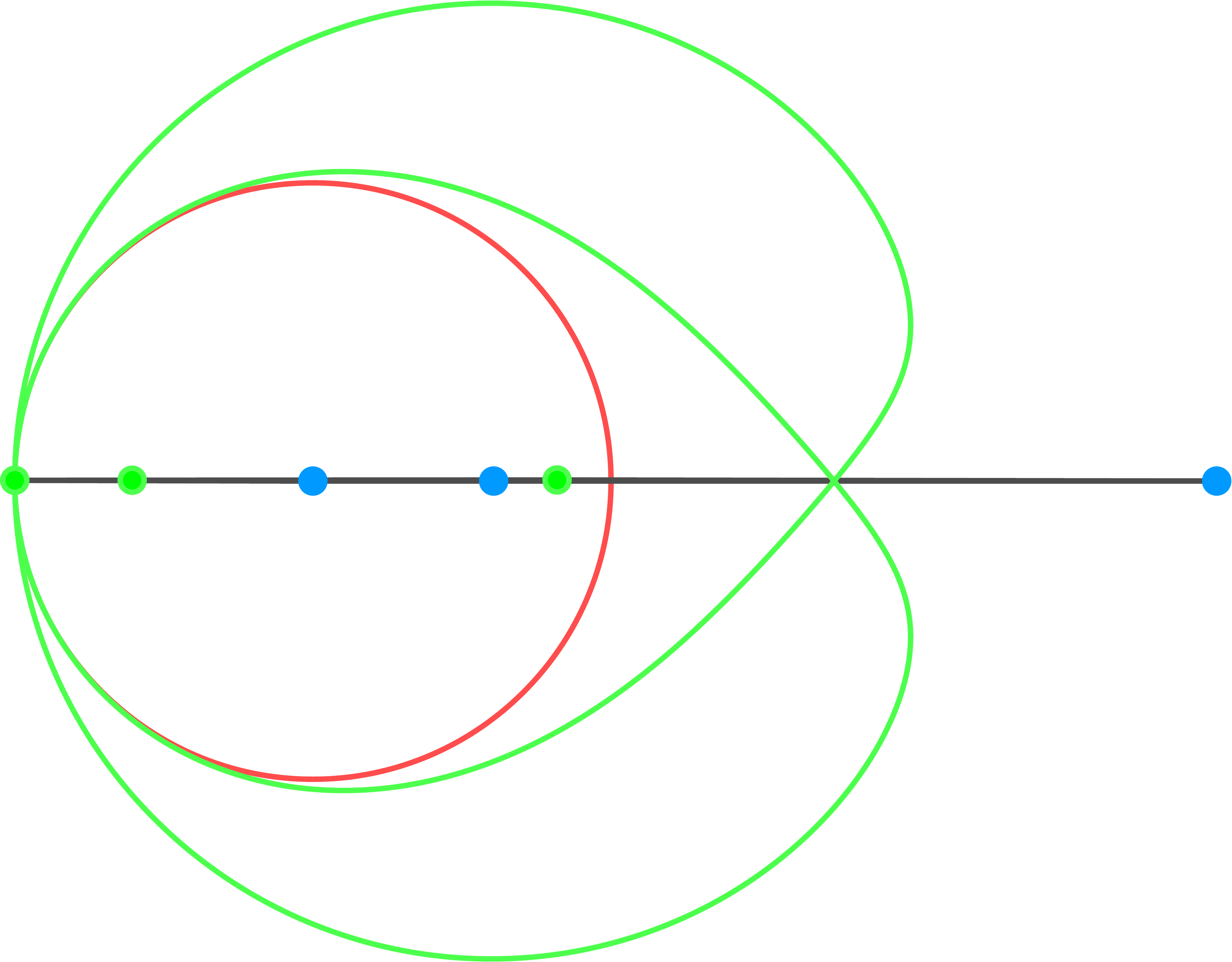}  
    \begin{scriptsize}
\put(-4.5,42){$m_1$}
\put(8,42){$m_2$}
\put(43,42){$m_3$}
\put(22.5,34){$M_1$}
\put(38,34){$M_2$}
\put(96.5,34){$M_3$}
\end{scriptsize}  
  \end{overpic} 
\qquad
  \begin{overpic}
    [width=50mm]{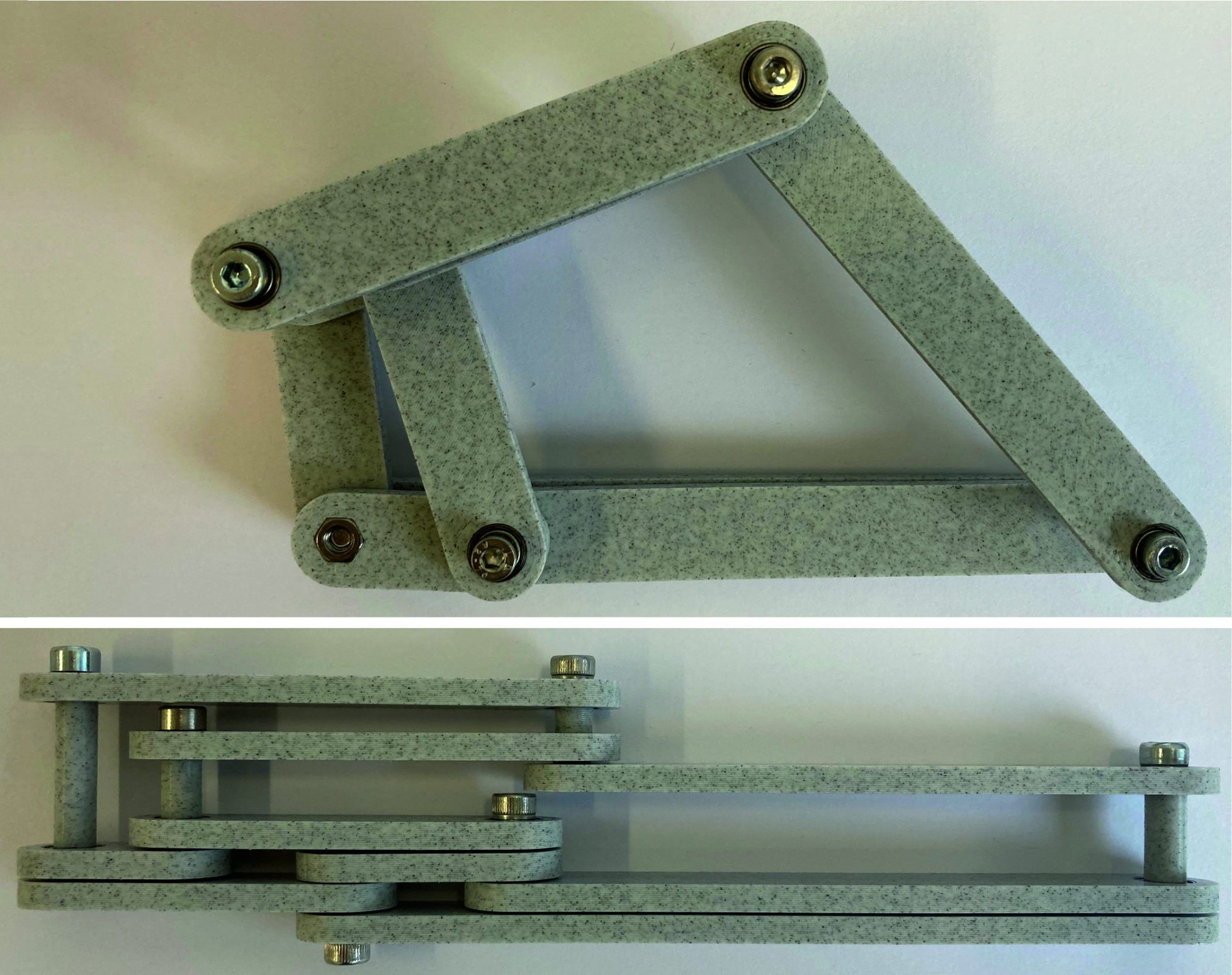}  
  \end{overpic} 
\end{center}	
\caption{(left) Visualization of the bar-joint framework with a flexion of order 23. The green curve shows the coupler curve of the point $m_1$, which results by giving away the first leg. The red curve is a circle with radius $r_1$ and midpoint $M_1$. The circle intersects the green coupler curve in $m_1$ with multiplicity 6. (right) Model of the bar-joint framework with flexion order 23, which is constructed using multiple layers (bottom). The model allows a large flexion as illustrated in the top. }
  \label{fig5}
\end{figure}

\section{Final remarks, open problems and future work}\label{sec:end}

In the paper we presented a global approach for a proper redefinition of higher-order flexibility and rigidity.
We only discussed planar frameworks, but the proposed algebraic method works for frameworks of any dimension. 
Especially, it is planned to apply the iterative procedure of
Section \ref{sec:computing} also to the spatial version of 3-RPR manipulators, which are Stewart--Gough platforms. Any such manipulator (interpreted as bar-body framework) has to have configurations with a $6^{th}$-order flexion, whose detailed investigation is dedicated to future research. Furthermore we are interested in the highest possible flexion order\footnote{According to Definition \ref{def3} its upper bound is 39 but a configuration only depends on 30 unknowns (up to Euclidean motions), which can be adjusted. From that one might expect a maximal flexion order of 30.} of  Stewart--Gough configurations and their computation.

Let us close the paper with the following list of final remarks and open problems:
\begin{enumerate}
\item
Note that the presented approach does not only work for bar-joint frameworks but it can be applied to any framework with algebraic joints; i.e.\ the relative position of two jointed rigid bodies can be described algebraically. 
But it remains open to extend it to frameworks with non-algebraic joints (cf.\ \cite{tarnai}).
\item
With our approach we were able to give a proper redefinition of higher-order flexibility and rigidity, but the computation of the 
associated $(k,n)$-flex(es) in dependence of the time parameter $t$ remains open and is dedicated to future research. 
We plan to solve this problem by means of tropical geometry and Puiseux series as this promising approach was already successfully used in \cite{Abhilash} 
for analyzing the configuration space of mechanisms.

Moreover, for this task we also want to generate further examples by following an idea of Stachel \cite{stachel_private} using the two-point guidance method, where the points $\Vkt x_1(t)$ and $\Vkt x_2(t)$ are in (higher-order) singularities of their paths at $t=0$. 
\item
Our approach operates over $\CC$ and does not take reality issues into account so far (for a local attempt see \cite{wampler}). For example, a planar 4-bar mechanism where the bar lengths $a,b,c,d$ fulfill the equation $a+b+c=d$ has a 1-dimensional set of configurations which are all complex with exception of one single configuration $\Vkt X$. 
Our algorithm would assign to $\Vkt X$ the flexion order $\infty$ but it is only shaky over $\RR$. 
We suggest the following procedure for resolving this minor problem\footnote{Note that this special case is circumvented by the formulation of Definition \ref{def3}, as we assumed that the considered configuration does not belong to a continuous flexion of the framework (over $\CC$).}.  
  Namely, instead of just assigning the value $\infty$ as flexion order, we propose to consider the corresponding $(k,n)$-flex(es) mentioned in item 2 above.  
	More precisely we are only interested in the degree $n^{\RR}\leq n$ of the highest possible real flexion. The(se) number(s) can then be used to assign a real flexion order to the configuration. 
  
  Note that the analysis of a framework configuration $\Vkt X$, which corresponds to an isolated real solution within a higher-dimensional complex configuration set, has to be
	handled with special care, as in this case $\Vkt X$ can also arise as an embedded component in the complex solution set (cf.\ \cite{zijia}). Then the irredundant primary decomposition proposed in the algorithm of Section \ref{sec:algorithm} is not unique anymore (cf.\ Remark \ref{rem:unique}). 
  The study of further examples in this context is dedicated to future research.  
\item 
The algorithm presented in Section \ref{sec:algorithm} for determining the intersection multiplicity requires global constructions (like primary decomposition and U-resultant method), but the multiplicity is a local property according to \cite{kirby}. 
Therefore again one can think about using local methods (e.g.\ Serre's Tor formula) 
to determine this number. It remains open if 
these local methods can also detect a continuous flexion and if they work in all cases (like the presented global approach). 
\end{enumerate}

\begin{acknowledgement} 
The author wants to thank to Hellmuth Stachel for the detailed discussions on his approach and for his permission to use Figure \ref{fig2}. Moreover, thanks to Daniel Huczala for building the model illustrated in Figure \ref{fig5}.  
Further, the author wants to thank the organizers of the Workshop ``Kinematic Aspects of Robotics'', which was part of the ``Special Semester on Rigidity and Flexibility'' held at the Johann Radon Institute for Computational and Applied Mathematics (Linz, Austria) in 2024, for invitation; particularly Zijia Li for fruitful discussions on primary decomposition.
\end{acknowledgement}

\appendix

\section*{Appendix: Polynomial of degree 32}\label{sec:app}
\vspace{-1mm}
\begin{small}
\begin{equation}
    \begin{split}
&516969488961264858296977044q_1^{32} - 9280309213987777419484380570q_1^{31}q_2 + \\ &43526270232117271834556502073q_1^{30}q_2^2 - 45280692730479399589412412168q_1^{29}q_2^3 - \\
&71413409266992435779029661320q_1^{28}q_2^4 + 733787582609859082926495640512q_1^{27}q_2^5 - \\
&1216057499416546331816336021712q_1^{26}q_2^6 + 1178525008268380508404672967040q_1^{25}q_2^7 + \\
&304983853881480483586054315776q_1^{24}q_2^8 + 373067534199906557276943674880q_1^{23}q_2^9 -  \\
&3506865857305108140637354422016q_1^{22}q_2^{10} + 1515457906293380496214847031296q_1^{21}q_2^{11} + \\
&2762451499211791028130610419712q_1^{20}q_2^{12} - 1507176820840441939654068420608q_1^{19}q_2^{13} - \\
&1140312255149192283851181674496q_1^{18}q_2^{14} + 370917717379345332121827704832q_1^{17}q_2^{15} + \\
&540356234313346392866675687424q_1^{16}q_2^{16} + 218622983025805473045891121152q_1^{15}q_2^{17} - \\
&513129700297250458379419975680q_1^{14}q_2^{18} - 146998314630604587702018375680q_1^{13}q_2^{19} + \\
&453229949991189146809689178112q_1^{12}q_2^{20} - 132638145759863692629486075904q_1^{11}q_2^{21} - \\
&148240985447636170928282402816q_1^{10}q_2^{22} + 124425897331594410107904983040q_1^9q_2^{23} - \\
&12386269734048188883819036672q_1^8q_2^{24} - 27049821097913736077418430464q_1^7q_2^{25} + \\
&14831418158089604670896996352q_1^6q_2^{26} - 1721669183596659665641930752q_1^5q_2^{27} - \\
&1309349875968694100160413696q_1^4q_2^{28} + 691975482131520534161129472q_1^3q_2^{29} - \\
&156210223994716269983039488q_1^2q_2^{30} + 18063680521134606070579200q_1q_2^{31} - \\
&874805860916262711853056q_2^{32}        =0 
    \end{split}
\end{equation}
\end{small}

\end{document}